\documentclass{article}
\usepackage[utf8]{inputenc}				
\usepackage{microtype}

\usepackage[numbers,square,sort&compress]{natbib}

\usepackage{caption}
\usepackage{subcaption}
\usepackage{amsmath}
\usepackage{amssymb}
\usepackage{bbm}
\usepackage{verbatim}
\usepackage{graphicx,epstopdf}
\usepackage{array}
\usepackage{tikz-cd}
\usepackage{xfrac}
\usepackage{faktor}

\usepackage{braket,amsfonts}
\usepackage{multirow}
\usepackage{graphbox}
\usepackage{overpic}
\usepackage{algorithm}
\usepackage{algorithmic}

\usepackage{graphicx}
\usepackage{amsmath,amsfonts,amsthm}
\usepackage{amssymb,latexsym}
\usepackage{fixmath}
\usepackage{mathrsfs,amsbsy}
\usepackage{enumerate}
\usepackage{algorithm,algorithmic}
\usepackage{kbordermatrix}
\usepackage{xcolor}
\usepackage{caption}
\usepackage{subcaption}

\usepackage{cleveref}   

\def\vextra{\vphantom{\vrule height0.40cm width0.9pt depth0.1cm}}

\def\bu{\pmb{u}}

\def\bw{\pmb{w}}
\def\bx{\pmb{x}}
\def\by{\pmb{y}}
\def\bz{\pmb{z}}

\def\wtd{\widetilde}
\def\what{\widehat}

\def\bbC{\mathbb{C}}
\def\bbH{\mathbb{H}}
\def\bbR{\mathbb{R}}

\def\ibbR{\iota\mathbb{R}}

\def\cI{{\cal I}}

\def\cN{{\cal N}}

\def\scrA{\mathscr{A}}
\def\scrB{\mathscr{B}}
\def\scrR{\mathscr{R}}
\def\scrF{\mathscr{F}}
\def\scrL{\mathscr{L}}

\DeclareMathOperator{\F}{F}
\DeclareMathOperator{\HH}{H}
\DeclareMathOperator{\T}{T}

\DeclareMathOperator{\eig}{eig}
\DeclareMathOperator{\NRes}{NRes}

\def\CARE{\mbox{\sc care}}

\def\FP{\mbox{\sc fp}}
\def\LOS{\mbox{\sc los}}
\def\LQR{\mbox{\sc lqr}}
\def\mNT{m\!\mbox{\sc nt}}
\def\NT{\mbox{\sc nt}}
\def\PSD{\mbox{\sc psd}}
\def\SDA{\mbox{\sc sda}}
\def\SDRE{\mbox{\sc sdre}}
\def\SCARE{\mbox{\sc scare}}

\def\SSDRE{\mbox{\sc ssdre}}
\def\SDC{\mbox{\sc sdc}}

\DeclareMathOperator{\rmc}{c}
\DeclareMathOperator{\dd}{d\!}
\DeclareMathOperator{\diag}{diag}
\DeclareMathOperator{\vex}{vec}

\date{today}
\numberwithin{equation}{section}

\newtheorem{theorem}{Theorem}[section]

\newtheorem{lemma}{Lemma}[section]

\theoremstyle{definition}
\newtheorem{assumption}{Assumption}[section]

\newtheorem{remark}{Remark}[section]
\newtheorem{example}{Example}[section]

\newcommand{\tr}{\textcolor{red}}

\numberwithin{algorithm}{section}

\allowdisplaybreaks

\title{Numerical Solutions for Stochastic Continuous-time Algebraic Riccati Equations}
\author{Tsung-Ming Huang\thanks{ Department of Mathematics,  National Taiwan Normal
    University, Taipei 116, Taiwan  ({\tt min@ntnu.edu.tw}).} \and
    Yueh-Cheng Kuo\thanks{Department of Mathematical Sciences, National Chengchi University, Taipei 116, Taiwan ({\tt 	kuoyc@nccu.edu.tw}).} \and
    Ren-Cang Li\thanks{ Department of Mathematics,
    University of Texas at Arlington, Arlington, USA ({\tt rcli@uta.edu}).} \and
    Wen-Wei Lin\thanks{ Nanjing Center for Applied Mathematics,
Nanjing, China; Department of Applied Mathematics, National Yang Ming Chiao Tung University, Hsinchu 300, Taiwan. ({\tt wwlin@math.nctu.edu.tw}).}}
\date{\today}

\begin{document}

\maketitle

\begin{abstract}
We are concerned with efficient numerical methods for  stochastic continuous-time algebraic Riccati equations (\SCARE). Such equations frequently arise from
the state-dependent Riccati equation approach which is perhaps the only systematic way today to study nonlinear control problems. Often involved Riccati-type equations are of small scale, but have to be solved repeatedly in real time. Important applications include the 3D missile/target engagement, the F16 aircraft flight control, and the quadrotor optimal control, to name a few. A new inner-outer iterative method
that combines the fixed-point strategy and the structure-preserving doubling algorithm (\SDA) is proposed. It is proved that the method is monotonically convergent, and in particular, taking the zero matrix as initial,
the method converges to the desired stabilizing solution. Previously,  Newton's method
has been called to solve \SCARE, but it was mostly investigated from its theoretic aspect than numerical aspect in terms of robust and efficient numerical implementation. For that reason, we
revisit Newton's method for \SCARE, focusing on how to calculate each Newton iterative step efficiently
so that Newton's method for \SCARE\ can become practical.
It is proposed to use our new inner-outer iterative method, which is provably
convergent, to provide critical initial starting points for Newton's method to ensure its convergence.
%
Finally several numerical experiments are conducted to validate the new method and robust implementation of Newton's method.
\end{abstract}

\section{Introduction}\label{sec:intro}
The state-dependent Riccati equation (\SDRE) approach mimics the well developed linear optimal control theory \cite{zhdg:1995}
to deal with nonlinear control problems. It was proposed in \cite{pear:1962} and has become very popular within the last three decades due to its simplicity and practical
effectiveness. The basic idea is to simply hide nonlinearity in a nonlinear control system via expressing the dynamics system
 in the same form as a linear time-invariant control system, except that the coefficient matrices
depends on the state vector instead of being constant, while minimizing a nonlinear cost function of a quadratic-like structure, and then adopt the same solution formulas in form for the feedback control from the linear optimal control theory
by freezing the state-dependent coefficient matrices at a particular state. The approach, also known as {\em extended linearization}, provides perhaps the only systematic way to study nonlinear control problems, albeit with suboptimal solutions \cite{alks:2023,balt:2007,cime:2008,neko:2019}. The \SDRE\ scheme manifests state and control weighting functions to ameliorate the overall performance \cite{cime:2012}, as well as, capabilities and potentials of other performance merits such as global asymptotic stability \cite{lixi:2019,lilc:2018,behe:2018}.

There are a number of practical and successful applications in literature of the \SDRE\ approach:
the differential \SDRE\  with impact angle guidance strategies \cite{lwhxwl:2023,nakm:2021a} for 3D pursuer/target traject tracking or interception engagement,
the finite-time \SDRE\  for F16 aircraft flight controls \cite{cpws:2022},
the \SDRE\  optimal control design for quadrotors for enhancing the robustness against unmodeled disturbances \cite{chhu:2022}, and
the \SDRE\  position/velocity controls for a high-speed vehicle \cite{abbe:1999a},
the \SDRE\ nonlinear optimal controller for a non-holonomic wheeled moving robot in a dynamic environment with moving obstacles \cite{asfo:2020}, and the \SDRE\ optimal control method for distance-based formation control of multiagent systems with energy constraints \cite{base:2020}, to name a few.
%

The nonlinear system models with state-dependent linear structure in these applications we just mentioned
contain no stochastic components  to counter uncertainties, which go against
real-world environments where noises always come into play. In order to better simulate real-world situations,
we may have to build  a stochastic component into the models and then solve them by an extension of the \SDRE\ approach,
which we will call {\em the stochastic \SDRE\ (\SSDRE) approach}.

%
%
%

The stochastic {\em state-dependent\/} linear-structured control system  in continuous-time subject to multiplicative white noises
is described as
\begin{align}\label{eq:DSys-SSDC}
\dd\bx(t) =\Big[ A(\bx) \bx + B(\bx) \bu\Big]\, \dd t + \sum_{i=1}^r \Big[A_0^i(\bx) \bx + B_0^i(\bx)\bu\Big] \dd \bw_i(t),
\end{align}
where $\bx\equiv \bx(t)$ and $\bu\equiv \bu(t)$ are the state vector and the control input, respectively,
$A(\bx),\,A_0^i(\bx) \in \bbR^{n \times n}$ and $B(\bx)$, $B_0^i(\bx) \in \bbR^{n \times m}$
for $i = 1, \ldots, r$ are the state-dependent coefficient (\SDC) matrices,
$\bw(t) = [ w_1(t), \cdots, w_r(t) ]^{\T}$ is a standard Wiener process vector whose entries $w_i(t)$ describe standard Brownian motions.
The word ``{\em state-dependent\/}'' reflects the fact that the coefficient matrices are matrix-valued function of  state vector $\bx$.
Associated with this linear-structured nonlinear dynamical system \eqref{eq:DSys-SSDC} is a quadratic-structured cost functional:
\begin{subequations}\label{eq:quad-func}
\begin{equation}\label{eq:quad-func-1}
    J(t_0, \bx_0; \bu) = E \left\{ \int_{t_0}^{\infty} \begin{bmatrix} \bx \\ \bu \end{bmatrix}^{\T}
    \begin{bmatrix}
        Q(\bx) & L(\bx) \\ L(\bx)^{\T} & R(\bx)
    \end{bmatrix}
    \begin{bmatrix}
        \bx \\ \bu
    \end{bmatrix} \dd t\right\},
\end{equation}
with respect to the state $\bx$ and control $\bu$ and subject to \eqref{eq:DSys-SSDC} with given initial value $\bx(t_0)=\bx_0$,
where $Q(\bx) \in \bbR^{n \times n}$, $L(\bx) \in \bbR^{n \times m}$, and $R(\bx) \in \bbR^{m \times m}$,
also depending on  state vector $\bx$, such that
\begin{equation}\label{eq:quad-func-2}
R(\bx)\succ 0
\quad\mbox{and}\quad
\begin{bmatrix}
        Q(\bx) & L(\bx) \\ L(\bx)^{\T} & R(\bx)
\end{bmatrix} \succeq 0,
\end{equation}
\end{subequations}
i.e., $R(\bx)$ is symmetric and positive definite and the bigger matrix is symmetric and positive semidefinite definite.
Equivalently, \eqref{eq:quad-func-2} is same as $R(\bx)\succ 0$ and $Q(\bx) - L(\bx) R(\bx)^{-1} L(\bx)^{\T} \succeq 0$.

When the \SDC\ matrices, including those in the cost functional, no longer depend on state $\bx$, the system
\eqref{eq:DSys-SSDC} with \eqref{eq:quad-func} becomes a stochastic linear time-invariant control system, whose control law
can be solved via the generalized type of algebraic Riccati equation \cite{damm:2004}. In the same spirit of the \SDRE\ approach
\cite{balt:2007,cime:2012,pear:1962},
mimicking the well-know \LQR\ (linear quadratic regulator) approach via the algebraic Riccati equation for the linear time-invariant control system,
the \SSDRE\ approach suboptimally
minimizes the cost functional \eqref{eq:quad-func-1} with state $\bx$ frozen and  computes a control law based on the solution
of the following {\em stochastic state-dependent continuous-time algebraic Riccati equation\/} 
in $X$:
\begin{align}
     & A(\bx)^{\T} X + X A(\bx) + \Pi_{11}(X) + Q(\bx) \nonumber \\
     & - \left[ X B(\bx) + \Pi_{12}(X) + L(\bx) \right]
      \left[ \Pi_{22}(X) + R(\bx)\right]^{-1}
     \left[ B(\bx) X + \Pi_{12}(X) + L(\bx) \right]^{\T} = 0. \label{eq:SSDCARE}
\end{align}
In the \SSDRE\ approach, \eqref{eq:SSDCARE} are repeatedly solved for numerous different frozen states $\bx$.

As far as solving \eqref{eq:SSDCARE}  for a particular frozen state $\bx$ is concerned, the role of $\bx$ is simply
used to determine the defining matrices to given a nonlinear matrix equation of the particular type.
For that reason, there is no loss of generality for us to suppress the dependency on $\bx$ and
focus on efficiently solving the nonlinear matrix equation
\begin{subequations}\label{eq:SCARE-0}
\begin{multline}\label{eq:SCARE-eq}
\scrR(X):=A^{\T}X+XA+Q+\Pi_{11}(X) \\
  -[XB+L+\Pi_{12}(X)][R+\Pi_{22}(X)]^{-1}[XB+L+\Pi_{12}(X)]^{\T}=0,
\end{multline}
where $\scrR(X)$ is defined as the rational matrix-valued function, the left-hand side of the equation,
and
\begin{align}\label{eq:SCARE-Pi}
\Pi(X)&:= \begin{bmatrix}
              \Pi_{11}(X) & \Pi_{12}(X) \\
              \Pi_{12}(X)^{\T} & \Pi_{22}(X)
            \end{bmatrix}
 := \begin{bmatrix}
              \sum\limits_{i=1}^r {A_0^{i}}^{\T}XA_0^i & \sum\limits_{i=1}^r {A_0^{i}}^{\T}XB_0^i \\
              \sum\limits_{i=1}^r {B_0^{i}}^{\T}XA_0^i & \sum\limits_{i=1}^r {B_0^{i}}^{\T}XB_0^i
            \end{bmatrix}\nonumber\\
         &=\begin{bmatrix}
              A_0^1 & B_0^1  \\
              \vdots &\vdots  \\
              A_0^r & B_0^r  \\
            \end{bmatrix}^{\T}\big(I_r\otimes X\big)  \begin{bmatrix}
              A_0^1 & B_0^1  \\
              \vdots &\vdots  \\
              A_0^r & B_0^r  \\
            \end{bmatrix},
\end{align}
Here, and in what follows, matrices $A,\,A_0^i,\, B,\,B_0^i,\, L,\, R$, etc. all have the same sizes as their state-dependent
counterparts above, and
\begin{equation}\label{eq:quad-func-const}
R\succ 0
\quad\mbox{and}\quad
\begin{bmatrix}
        Q & L \\ L^{\T} & R
\end{bmatrix} \succeq 0.
\end{equation}
\end{subequations}
It can be verified that \eqref{eq:quad-func-const} is equivalent to
\begin{equation}\tag{\ref{eq:quad-func-const}'}
R\succ 0
\quad\mbox{and}\quad
Q - L R^{-1} L^{\T}\succeq 0.
\end{equation}

We call \eqref{eq:SCARE-0} {\em the stochastic continuous-time algebraic Riccati equation\/} (\SCARE).
It takes exactly the same form as
the generalized type of algebraic Riccati equation in the case for the
stochastic linear time-invariant control problem \cite{damm:2004}.
From the perspective of mathematics, immediately two major  questions arises:
\begin{enumerate}[1)]
  \item Does \SCARE\ \eqref{eq:SCARE-eq}, as a nonlinear matrix equation, have a solution?
        If  it does, is there a solution that satisfies the need of any desired control law?
  \item How to efficiently solve \SCARE\ \eqref{eq:SCARE-eq} for the desired solution?
\end{enumerate}
Our task in this paper is to fully address the second question. But first we shall state an answer to the first question
from the literature.
To streamline the notation, we introduce
$A_0^0=A$ and $B_0^0=B$.
Two assumptions in the terminology of control are needed:

\begin{assumption}\label{asm:stab}
The pair $(\{ A_0^i \}_{i=0}^r, \{ B_0^i \}_{i=0}^r)$
          is stabilizable {\cite[Definition 4.1.2]{drms:2013}}, namely, there exists $F \in \bbR^{m \times n}$
           such that the linear differential equation for $S\equiv S(t)$:
           $$
           \frac{\dd}{\dd t}S=\scrL_{F}S:=(A + BF) S + S (A + BF)^{\T} + \sum_{i=1}^r(A_0^i + B_0^iF) S(A_0^i + B_0^iF)^{\T},
          $$
          is exponentially stable, or equivalently, the evolution operator $e^{\scrL_{F}(t-t_0)}$ is exponentially stable.
\end{assumption}

\begin{assumption}\label{asm:dete}
The pair $(  C,\{ \wtd A_0^i\}_{i=0}^r)$ is detectable {\cite[Definition 4.1.2]{drms:2013}}, where $C\in \mathbb{R}^{p\times n}$
          satisfies $C^{\T} C = Q - L R^{-1} L^{\T}$ and $\wtd A_0^i:= A_0^i-B_0^iR^{-1}L^{\T}$ for $0\le i\le r$, namely, there exists $K \in \bbR^{n \times p}$
          such that the linear differential equation for $S\equiv S(t)$:
          $$
          \frac{\dd}{\dd t}S=\scrL^{K}S:=
          (\widetilde{A}_0^0 + KC) S + S (\widetilde{A}_0^0 + KC)^{\T} + \sum_{i=1}^r \widetilde{A}_0^i S\,\big[\wtd A_0^{i}\big]^{\T},
          $$
          is exponentially stable.
\end{assumption}


%
%
%
%

We can now state the major theoretic result on the existence of a solution to \SCARE\ \eqref{eq:SCARE-0}.

\begin{theorem}[{\cite[Theorem 5.5.3]{drms:2013}}]\label{thm:existence}
If both Assumptions~\ref{asm:stab} and \ref{asm:dete} hold, then
\SCARE\ \eqref{eq:SCARE-0}
has a unique positive semi-definite (\PSD) solution $X_*$, which is also stabilizing,  such that
$$
\big(A + B F_*,A_0^1 + B_0^1 F_*, \cdots, A_0^r + B_0^r F_*\big)
$$
is stable, i.e., the linear differential equation $\frac{\dd}{\dd t}S(t)=\scrL_{F_*}S(t)$ is exponentially stable, where
\begin{align}
    F_* = - \Big[ \sum_{i=1}^r {B_0^i}^{\T} X_* B_0^i + R \Big]^{-1}
            \Big[ B^{\T} X_* + \sum_{i=1}^r {B_0^i}^{\T} X_* A_0^i + L^{\T}\Big]. \label{eq:F_X*}
\end{align}
\end{theorem}

This theorem stipulates a sufficient condition for \SCARE\ \eqref{eq:SCARE-0} to have a unique  \PSD\ solution, which is also stabilizing. In applications, this is the solution of interest.
%
%
%
%
%
Existing methods to solve \SCARE\ \eqref{eq:SCARE-0} for its stabilizing solution include
Newton's  method (\NT) \cite{dahi:2001},
modified Newton's  method  \cite{chll:2011,guo:2002a,ivan:2007}
(for a special case of \eqref{eq:SCARE-0}),  the fixed-point (\FP) iteration \cite{guli:2023}
and a structure-preserving double algorithm (\SDA) \cite{guli:2023}.
Newton's method or its variants are often the defaults when it comes to solve a nonlinear equation but
they usually need sufficiently accurate starting points to begin with, and \FP\ iterations are usually slowly convergent.
The \SDA\ in \cite{guli:2023} is built upon a newly introduced operation which
induces an unappealing side-effect of doubling the sizes of the matrices per \SDA\ iterative step and
hence the method is more of theoretic interest than practical
one, despite of its mathematical elegance.
The main task of this paper is to develop efficient and reliable algorithms for solving \SCARE\ \eqref{eq:SCARE-0}, for
small and modest $n$. It is noted that \SCARE\ with small $n$ (up to a coupe of tens) are surprisingly common in real-world applications
(see, e.g., \cite{chhu:2022,cpws:2022,lwhxwl:2023,nakm:2021a} and references therein)
than one might think.
Specifically,
we consider the following two types of methods for solving \SCARE.
\begin{enumerate}[(i)]
\item \FP\SDA: \SCARE\ \eqref{eq:SCARE-0} can be rewritten as
      \begin{align}
      \scrR(X) := \big[A_{\rmc}(X)\big]^{\T} X + X A_{\rmc}(X) - X G_{\rmc}(X) X + H_{\rmc}(X) = 0 \label{eq:SCARE-CARE-0}
      \end{align}
      where $A_{\rmc}(X), G_{\rmc}(X)$ and $H_{\rmc}(X)$ are matrix-valued functions to be defined
      in \Cref{sec:prelim}. It has the form of a continuous algebraic Riccati equation (\CARE) from the optimal control theory \cite{hull:2018},
      except that here $A_{\rmc}(X), G_{\rmc}(X)$ and $H_{\rmc}(X)$ depend on the solution but are not constant.
      Hence if $A_{\rmc}(X), G_{\rmc}(X)$ and $H_{\rmc}(X)$ are freezed at an approximate solution, then \eqref{eq:SCARE-CARE-0} becomes
      a \CARE\ which may be solved by the structured-preserved doubling algorithm \SDA\ \cite[Chapter 5]{hull:2018} for hopefully
      a better approximate solution. The process repeats until convergence.
      We will investigate this method in \Cref{sec:FP-CARE-SDA}, including its practical implementation and convergence analysis.
      It will be shown that the method always converges to a stabilizing solution, starting from the initial approximation $X=0$.
      This is a new method.


%
\item Newton's method: \SCARE\ \eqref{eq:SCARE-0} is a nonlinear equation.  It is natural to solve it with Newton's method, and in fact
      that has been done in \cite{dahi:2001} mostly from the theoretic side of Newton's method in terms of the actual defining equation
      for Newton's iterative step and its convergence analysis rather than its effective and robust implementation from the practical side.
      We will revisit Newton's method for a practical implementation in \Cref{sec:Newton}.

\end{enumerate}
      Despite that generically Newton's method is eventually quadratically convergent, there are two major obstacles in the case of
      \SCARE\ \eqref{eq:SCARE-0}. The first obstacle is universal: Newton's method needs a sufficiently good initial approximation
      to ensure convergence, and the second one is particular to \SCARE\ \eqref{eq:SCARE-0}: the matrix equation for
      Newton's iterative step is challenging numerically unless $n$ is very small (up to a couple of tens). We will provide practical solutions to both obstacles in the case of \SCARE\ \eqref{eq:SCARE-0}.
      The first obstacle can be tackled by running \FP\SDA\ to calculate a decent initial approximation to feed into
      Newton's method, while for the second obstacle when $n$ is very small so that $n^2$ is modest,
      the  defining matrix equation for Newton's iterative step can be transformed into the standard form of a linear system of size $n^2$-by-$n^2$
      and the latter can then be solved by the Gaussian elimination. But when $n^2$ is too large (up to tens of thousands), we may have to resort some iterative schemes,
      one of which is to combining the fixed-point technique with Lyapunov equation solving to iteratively
      compute the Newton step. Whether this scheme for Newton steps is convergent or not is yet another problem to worry about.

This paper is organized as follows. In Section~\ref{sec:prelim}, we introduce some important properties of \SCARE.  In Section~\ref{sec:FP-CARE-SDA}, we propose a new inner-outer iterative method \FP\SDA, combining the fixed-point strategy with \SDA\ to solve \SCARE\ and conduct a convergence analysis of \FP\SDA. We revisit Newton's method for \SCARE, focusing on how to calculate each Newton iterative step efficiently in Section~\ref{sec:Newton}. Numerical results are presented to demonstrate the effective ness of our new method \FP\SDA\ and
our implementations of Newton's method in Section~\ref{sec:egs}. Finally, the conclusion is drawn in Section~\ref{sec:conclusion}.
There are three appendix sections. \Cref{apx:FP} reviews the \FP\ iteration in \cite{guli:2023} and compares it
with an \FP\ iteration based on the construction of the first standard form (SF1) in \cite{hull:2018}.
\Cref{sec:CARE-SDA} reviews \SDA\ for \CARE, tailored for the inner-iteration of our \FP\SDA.
In \Cref{sec:mNewton} we investigate a modified Newton's method for \SCARE\ as an extension of the one in Guo~\cite{guo:2002a}.

%
%

\smallskip\noindent
{\bf Notation.}
$\bbR^{n\times m}$ is the set of $n$-by-$m$ real matrices, $\bbR^n=\bbR^{n\times 1}$ and
$\bbR=\bbR^1$, and their complex counterparts are $\bbC^{n\times m}$, $\bbC^n$ and $\bbC$.
Finally, $\bbC_-$ stands for the set of all complex numbers in the left half of the complex plane, and $\iota$ is the imaginary unit.
$I_n\in\bbR^{n\times n}$ is the $n$-by-$n$ identity matrix.

$\bbH^{n\times n}\subset\bbR^{n\times n}$ is the set of $n$-by-$n$ real symmetric matrices.
For $X\in \bbH^{n\times n}$,
$X\succeq 0$ ($X\succ 0$) means that $X$ is positive semidefinite (positive definite), and
$X\preceq 0$ ($X\prec 0$) means $-X\succeq 0$ ($-X\succ 0$). This introduces
a partial order in $\bbH^{n\times n}$: $X\succeq Y$ ($X\succ Y$) if $X-Y\succeq 0$ ($X-Y\succ 0$) and vice versa.

Given a matrix/vector $B$, $B^{\T}$ and $B^{\HH}$ denote its transpose and the complex conjugate transpose, respectively.
$\|B\|_p$ and $\|B\|_{\F}$ are the $\ell_p$ operator norm and Frobenius norm of $B$, respectively,
and $\cN(B)$ denotes its null space.
If $B$ is also square, $\eig(B)$ is the multiset of the eigenvalues of $B$.
Other notations will be explained at their first appearances.

\section{Properties of \SCARE}\label{sec:prelim}
In this section, we will establish a few preliminary results related to \SCARE\ \eqref{eq:SCARE-0} to set up the stage
for our algorithmic design and convergence analysis throughout the rest of this paper.
It is assumed that \eqref{eq:quad-func-const} holds.
Define 
\begin{subequations}\label{eq:c-matrices}
\begin{equation}\label{eq:SCARE-1b}
     L_{\rmc}(X) = L +  \Pi_{12}(X), \quad R_{\rmc}(X) = R+\Pi_{22}(X), \quad Q_{\rmc}(X) = Q+\Pi_{11}(X).
\end{equation}
and
\begin{align}
      A_{\rmc}(X) &= A - B\big[R_{\rmc}(X)\big]^{-1} \big[L_{\rmc}(X)\big]^{\T}, \label{eq:SCARE-Ac}\\
      G_{\rmc}(X) &= B \big[R_{\rmc}(X)\big]^{-1} B^{\T}, \label{eq:SCARE-Gc}\\
      H_{\rmc}(X) &= Q_{\rmc}(X) - L_{\rmc}(X) \big[R_{\rmc}(X)\big]^{-1} \big[L_{\rmc}(X)\big]^{\T}. \label{eq:SCARE-Hc}
\end{align}
\end{subequations}
An equivalent form of \eqref{eq:SCARE-0} is
\begin{equation}\label{eq:SCARE-CARE-1}
  \scrR(X) := \big[A_{\rmc}(X)\big]^{\T} X + X A_{\rmc}(X) - X G_{\rmc}(X) X + H_{\rmc}(X) = 0.
\end{equation}
Frequently, we will use this fact:
$\Pi(X)\succeq 0$ for any $X\succeq 0$ and
\begin{align}
      \Pi(X) \succeq \Pi(Y) \succeq 0\quad \mbox{for} \quad X \succeq Y \succeq 0. \label{eq:assumption_Pi}
\end{align}

\begin{lemma} \label{lm:stabilizable}
If $(A,B)$ is stabilizable,
then $\big(A_{\rmc}\big(\wtd X\big), G_{\rmc}\big(\wtd X\big)\big)$ is stabilizable for any $\wtd X\succeq 0$.
\end{lemma}

\begin{proof}
Write $\wtd A=A_{\rmc}\big(\wtd X\big)$, $\wtd G=G_{\rmc}\big(\wtd X\big)$, $\wtd L=L_{\rmc}\big(\wtd X\big)$, and $\wtd R=R_{\rmc}\big(\wtd X\big)$.
It suffices to show that for any $\lambda\in\bbC$ with its real part $\Re(\lambda) \ge 0$ \cite[p.50]{zhdg:1995}
$$
\by\in\bbC^n,\,\by^{\HH} \left[ \wtd A-\lambda I,\,\wtd G\right]=\by^{\HH} \left[A - B \wtd R^{-1} \wtd L,\, B \wtd R^{-1} B^{\T}\right] = 0
\quad\Rightarrow\quad \by=0.
$$
In fact, $\by^{\HH}\wtd G = \by^{\HH} B \wtd R^{-1} B^{\T} = 0$ implies
$\by^{\HH}\wtd G\by = (\wtd R^{-1/2} B^{\T}\by)^{\HH} (\wtd R^{-1/2} B^{\T}\by) = 0$ which leads to
$\by^{\HH}B\wtd R^{-1/2}=0$, yielding $\by^{\HH}B=0$.
At the same time
$$
0=\by^{\HH} ( \wtd A-\lambda I)= \by^{\HH} (A - B \wtd R^{-1} \wtd L-\lambda I)=\by^{\HH} (A-\lambda I).
$$
Hence $\by^{\HH}[A-\lambda I,\,B]=0$ where $\Re(\lambda) \ge 0$, which implies $\by=0$ because $(A,B)$ is assumed stabilizable
\cite[p.50]{zhdg:1995}.
\end{proof}

\begin{lemma} \label{lm:detectable}
Suppose that
\begin{equation}\label{eq:ker-cond}
     \cN(Q - L R^{-1} L^{\T}) \subseteq  \cN(L)\bigcap\,\bigcap_{i=1}^r\cN(A_0^i).
\end{equation}
If $(Q - L R^{-1} L^{\T}, A)$ is detectable,
then, for any $\wtd X\succeq 0$,  $\big(H_{\rmc}\big(\wtd X\big), A_{\rmc}\big(\wtd X\big)\big)$ is detectable and $H_{\rmc}\big(\wtd X\big)\succeq 0$.
\end{lemma}

\begin{proof}
Since $Q - L R^{-1} L^{\T}\succeq 0$ by (\ref{eq:quad-func-const}'), there is a matrix $C$ such that $C^{\T} C= Q - L R^{-1} L^{\T}$. We know that $\cN(C)=\cN(Q - L R^{-1} L^{\T})$, and $(Q - L R^{-1} L^{\T}, A)$ is detectable if and only if $(C, A)$ is detectable.

Write $\wtd A=A_{\rmc}\big(\wtd X\big)$, $\wtd G=G_{\rmc}\big(\wtd X\big)$, $\wtd L=L_{\rmc}\big(\wtd X\big)$,
$\wtd R=R_{\rmc}\big(\wtd X\big)$, and
$$
\wtd H=H_{\rmc}\big(\wtd X\big)
      =C^{\T} C+\underbrace{\big[Q-C^{\T} C+\Pi_{11}\big(\wtd X\big)\big] - \wtd L \wtd R^{-1}  \wtd L^{\T}}_{=:\wtd\Lambda}.
$$
We claim that $\wtd\Lambda\succeq 0$. This is because  $\Pi\big(\wtd X\big)\succeq 0$ and
$$
\begin{bmatrix}
Q-C^{\T} C & L \\ L^{\T} & R
\end{bmatrix}
=\begin{bmatrix}
 L R^{-1} L^{\T} & L \\ L^{\T} & R
\end{bmatrix}
=\begin{bmatrix}
   I & LR^{-1} \\
   0 & I
 \end{bmatrix}\begin{bmatrix}
                0 & 0 \\
                0 & R
              \end{bmatrix}\begin{bmatrix}
   I & LR^{-1} \\
   0 & I
 \end{bmatrix}^{\T}\succeq 0,
$$
and hence we get
     \begin{align*}
            0 \preceq\Pi\big(\wtd X\big) + \begin{bmatrix}
                 Q-C^{\T} C & L \\ L^{\T} & R
            \end{bmatrix} = \begin{bmatrix}
                 I & \wtd L \wtd R^{-1} \\ 0 & I
            \end{bmatrix} \begin{bmatrix}
                 \wtd\Lambda & 0 \\ 0 & \wtd R
            \end{bmatrix} \begin{bmatrix}
                 I & 0 \\ \wtd R^{-1} \wtd L^{\T} & I
            \end{bmatrix},
     \end{align*}
which implies that $\wtd\Lambda\succeq 0$   and $\wtd H = C^{\T} C + \wtd\Lambda \succeq 0$.
Now we are ready to show that $(\wtd H, \wtd A)$ is detectable. For that purpose, it suffices to show
that  for any $\lambda\in\bbC$ with its real part $\Re(\lambda) \ge 0$ \cite[p.52]{zhdg:1995}
$$
\bz\in\bbC^n,\,\,
\begin{bmatrix}
\wtd A-\lambda I \\
\wtd H
\end{bmatrix}\bz =0
\quad\Rightarrow\quad \bz=0.
$$
In fact $\wtd H \bz = 0$ implies $\bz^{\HH}\wtd H\bz=0$, i.e.,
\begin{align}\label{eq:detectable:pf-1}
     \bz^{\HH} C^{\T} C \bz + \bz^{\HH} \wtd\Lambda \bz = 0.
\end{align}
Hence, by \eqref{eq:detectable:pf-1},  we have $\bz^{\HH} C^{\T} C \bz=0$  and $\bz^{\HH} \wtd\Lambda \bz = 0$.
It follows from $\bz^{\HH} C^{\T} C \bz=0$ that $C \bz = 0$, which together with \eqref{eq:ker-cond} lead to
$A_0^i \bz = 0$ for $i = 1, \ldots, r$. Therefore, we have
$$
0=\big[\wtd A-\lambda I\big]\bz
 =\Big[A - B\wtd R^{-1} \Big(L+\sum_{i=1}^r {B_0^i}^{\T} \wtd X A_0^i\Big)\Big]\bz
 =\big[A-\lambda I\big]\bz
$$
and, together with $C\bz=0$, we get
$$
\begin{bmatrix}
A-\lambda I \\
C
\end{bmatrix}\bz =0
\quad\mbox{with $\Re(\lambda) \ge 0$},
$$
which implies $\bz=0$ because $(C, A)$ is assumed detectable.
\end{proof}


Define
\begin{align}
\Gamma(X):= \left[\begin{array}{cc|c}
              0&-A&-B\\
              -A^{\T}&Q_{\rmc}(X) & L_{\rmc}(X) \\ \hline
              -B^{\T}&\big[L_{\rmc}(X)\big]^{\T} & R_{\rmc}(X) \vextra
            \end{array}\right]
    =: \left[\begin{array}{c|c}
         \Gamma_{11}(X) & \Gamma_{12}(X) \\ \hline
         \big[\Gamma_{12}(X)\big]^{\T} & R_{\rmc}(X) \vextra
    \end{array}\right]  \in \bbR^{(2n+m)\times (2n+m)}, \label{eq:mtx_Gamma}
\end{align}
where $Q_{\rmc}(X)$, $L_{\rmc}(X)$, and $R_{\rmc}(X)$ are as in \eqref{eq:SCARE-1b}.
The Schur complement of $R_{\rmc}(X)$ in $\Gamma(X)$ is given by
\begin{align}\label{eq:Omega-dfn}
\Omega(X):=&\begin{bmatrix}
             0&-A\\
             -A^{\T}&Q_{\rmc}(X)
            \end{bmatrix}-
            \begin{bmatrix}
              -B\\
            L_{\rmc}(X)
            \end{bmatrix}
            \big[R_{\rmc}(X)\big]^{-1}
            \begin{bmatrix}
              -B\\
            L_{\rmc}(X)
            \end{bmatrix}^{\T}              \nonumber\\
   =& \begin{bmatrix}
        -B\big[R_{\rmc}(X)\big]^{-1}B^{\T}&-(A-B\big[R_{\rmc}(X)\big]^{-1} \big[L_{\rmc}(X)\big]^{\T})\\
     -(A^{\T}- L_{\rmc}(X)\big[R_{\rmc}(X)\big]^{-1}B^{\T})&Q_{\rmc}(X)- L_{\rmc}(X)\big[R_{\rmc}(X)\big]^{-1}\big[L_{\rmc}(X)\big]^{\T}  \end{bmatrix} \nonumber\\
     =&\begin{bmatrix}
        -G_{\rmc}(X) & -A_{\rmc}(X)\\
     -\big[A_{\rmc}(X)\big]^{\T} & H_{\rmc}(X)  \end{bmatrix},
\end{align}
where $A_{\rmc}(X)$, $G_{\rmc}(X)$ and $H_{\rmc}(X)$ are as in \eqref{eq:c-matrices}.

\begin{lemma} \label{lm:Omega-incr}
Let $\Omega(X)$ be defined as in \eqref{eq:Omega-dfn}. We have
\begin{align}\label{eq:increasing_Omega}
     \Omega(X)\succeq \Omega(Y)\text{ for } X\succeq Y\succeq 0.
\end{align}
\end{lemma}

\begin{proof}
Suppose that $X\succeq Y\succeq 0$. By \eqref{eq:mtx_Gamma}, \eqref{eq:SCARE-1b} and \eqref{eq:SCARE-Pi}, we have
\begin{align*}
\Gamma(X)&=\left[\begin{array}{c|cc}
              0&-A&-B\\\hline
              -A^{\T}&Q & L \vextra \\
              -B^{\T}&L^{\T} & R
    \end{array}\right]+\left[\begin{array}{ccc}
              0&\cdots&0\\\hline
              A_0^{1\T}&\cdots & A_0^{r\T} \vextra \\
              B_0^{1\T}&\cdots & B_0^{r\T}
            \end{array}\right] (I_r\otimes X)\left[\begin{array}{c|cc}
              0&A_0^1&B_0^1\\
              \vdots&\vdots & \vdots \\
              0&A_0^r & B_0^{r}
            \end{array}\right]\\
            &\succeq\left[\begin{array}{c|cc}
              0&-A&-B\\\hline
              -A^{\T}&Q & L \vextra\\
              -B^{\T}&L^{\T} & R
    \end{array}\right]+\left[\begin{array}{ccc}
              0&\cdots&0\\\hline
              A_0^{1\T}&\cdots & A_0^{r\T} \vextra\\
              B_0^{1\T}&\cdots & B_0^{r\T}
            \end{array}\right] (I_r\otimes Y)\left[\begin{array}{c|cc}
              0&A_0^1&B_0^1\\
              \vdots&\vdots & \vdots \\
              0&A_0^r & B_0^{r}
            \end{array}\right]=\Gamma(Y).
\end{align*}
Since $\Gamma(X)\succeq \Gamma(Y)$, we have for any $\bw\in \bbR^{2n}$,
\begin{align}
\Gamma_{\bw}(X)&:=\begin{bmatrix}
                   \bw &  \\
                    & I_m
                 \end{bmatrix}^{\T}\Gamma(X)\begin{bmatrix}
                   \bw &  \\
                    & I_m
                 \end{bmatrix}
   =\begin{bmatrix}
        \bw^{\T}\Gamma_{11}(X)\bw &\bw^{\T}\Gamma_{12}(X)   \\
        \big[\Gamma_{12}(X)\big]^{\T}\bw  & R_{\rmc}(X)
    \end{bmatrix} \nonumber\\
  & \succeq \begin{bmatrix}
                   \bw &  \\
                    & I_m
                 \end{bmatrix}^{\T}\Gamma(Y)\begin{bmatrix}
                   \bw &  \\
                    & I_m
                 \end{bmatrix}
  =\begin{bmatrix}
        \bw^{\T}\Gamma_{11}(Y)\bw &\bw^{\T}\Gamma_{12}(Y)   \\
        \big[\Gamma_{12}(Y)\big]^{\T}\bw  & R_{\rmc}(Y)
    \end{bmatrix}
    =: \Gamma_{\bw}(Y). \label{eq:Omega-incr:pf-1}
\end{align}
It can be seen that $\bw^{\T}\Omega(X)\bw$ and $\bw^{\T}\Omega(Y)\bw$ are the Schur complement of $R_{\rmc}(X)$ in $\Gamma_{\bw}(X)$
and that of $R_{\rmc}(Y)$ in $\Gamma_{\bw}(Y)$, respectively.
Let $\alpha>0$ such that
$
        \bw^{\T}\Omega(Y)\bw+\alpha>0,
$
and define
\begin{align}\label{eq:Omega-incr:pf-2}
\Gamma_{\bw}^{\alpha}(X):=\Gamma_{\bw}(X)+\left[\begin{array}{cc}
            \alpha &0  \\
            0 & 0
        \end{array}\right],\quad
\Gamma_{\bw}^{\alpha}(Y):=\Gamma_{\bw}(Y)+\left[\begin{array}{cc}
            \alpha &0  \\
            0 & 0
        \end{array}\right].
\end{align}
It can also be seen that $\bw^{\T}\Omega(Y)\bw+\alpha$ is the Schur complement of $R_{\rmc}(Y)$ in $\Gamma_{\bw}^{\alpha}(Y)$.
Since $R_{\rmc}(Y)\succ 0$ and $\bw^{\T}\Omega(Y)\bw+\alpha>0$, it follows from \eqref{eq:Omega-incr:pf-1} that
    $\Gamma_{\bw}^{\alpha}(X)\succeq \Gamma_{\bw}^{\alpha}(Y)\succ 0$. Hence,
\begin{align}\label{eq:Omega-incr:pf-3}
        0\prec \big[\Gamma_{\bw}^{\alpha}(X)\big]^{-1}\preceq \big[\Gamma_{\bw}^{\alpha}(Y)\big]^{-1}.
\end{align}
From \eqref{eq:Omega-incr:pf-2}, we see that the $(1,1)$st entries
of $\big[\Gamma_{\bw}^{\alpha}(X)\big]^{-1}$ and $\big[\Gamma_{\bw}^{\alpha}(Y)\big]^{-1}$
are $\big[\bw^{\T}\Omega(X)\bw+\alpha\big]^{-1}$ and $\big[\bw^{\T}\Omega(Y)\bw+\alpha\big]^{-1}$, respectively,
yielding, by \eqref{eq:Omega-incr:pf-3}, that
$$
0<\big[\bw^{\T}\Omega(X)\bw+\alpha\big]^{-1}\le \big[\bw^{\T}\Omega(Y)\bw+\alpha\big]^{-1}
$$
for any $\bw \in \bbR^{2n}$, and hence
$$
\bw^{\T}\Omega(X)\bw+\alpha\ge \bw^{\T}\Omega(Y)\bw+\alpha
\,\Rightarrow\,
\bw^{\T}\Omega(X)\bw\ge \bw^{\T}\Omega(Y)\bw
\,\Rightarrow\,
\Omega(X)\succeq \Omega(Y),
$$
as was to be shown.
\end{proof}

In terms of $\Omega(\cdot)$, \SCARE\ \eqref{eq:SCARE-0} has a new formulation.

\begin{lemma}\label{lm:SCARE-equiv}
Let $\Omega(X)$ be defined as in \eqref{eq:Omega-dfn}. \SCARE\ \eqref{eq:SCARE-0} can be reformulated as
\begin{align}\label{eq:SCARE-2}
    \begin{bmatrix} X &  -I \end{bmatrix}
    \Omega(X)
    \begin{bmatrix}
        X\\
     -I  \end{bmatrix}=0.
\end{align}
\end{lemma}

\section{Fixed-point iteration via \SDA} \label{sec:FP-CARE-SDA}
In this section, we propose our main algorithm that uses the structured-preserving doubling algorithm (\SDA) outlined in \Cref{alg:CARE-SDA} of \Cref{sec:CARE-SDA} as its workhorse
to solve \SCARE\ \eqref{eq:SCARE-0}. The basic idea is actually very simple. Recall its equivalent form
\eqref{eq:SCARE-CARE-1} of \eqref{eq:SCARE-0}, which appears in the form of \CARE\ \eqref{eq:CARE}, except
the defining coefficient matrices are solution-dependent. Hence, the fixed-point strategy naturally applies.
Given an approximate solution to \eqref{eq:SCARE-CARE-1}, we evaluate and freeze its coefficient matrices at the given approximation
to yield a \CARE\ which is then solved by \SDA\ for hopefully a better approximate solution.
The process repeats itself until convergence. This new method combines the fixed-point strategy and \SDA. For that reason, we name
the method
the {\em fixed-point iteration via \SDA} (\FP\SDA).

The key question is whether the process is convergent, and if it does, what is being converged to.
Under Assumptions~\ref{asm:stab} and \ref{asm:dete}, \eqref{eq:quad-func-const},
\eqref{eq:ker-cond}, and that $(A,B)$ is stabilizable and $(Q - L R^{-1} L^{\T}, A)$ is detectable, we will show that
the process indeed
converges monotonically to the stabilizing solution of \SCARE\ \eqref{eq:SCARE-0}.

\subsection{\FP\SDA}
Given an approximate solution $X_k$ to \eqref{eq:SCARE-CARE-1},
we freeze its solution-dependent coefficient matrices at $X=X_k$ to
get the following \CARE
\begin{subequations}\label{eq:FPSCARE(k)}
\begin{equation}\label{eq:FPSCARE(k)-eq}
      A_k^{\T} X + X A_k - X G_k X + H_k = 0,
\end{equation}
where
\begin{equation}\label{eq:FPSCARE(k)-coeffs}
      R_k = R_{\rmc}(X_k), \quad G_k = G_{\rmc}(X_k), \quad A_k = A_{\rmc}(X_k), \quad H_k = H_{\rmc}(X_k),
\end{equation}
\end{subequations}
and then we call \Cref{alg:CARE-SDA} to solve \eqref{eq:FPSCARE(k)} for its stabilizing solution $X_{k+1}$, if
it exists. This is our \FP\SDA\ as detailed in \Cref{alg:SCARE-SDA}.
In order for  \CARE\ \eqref{eq:FPSCARE(k)} to be well-defined, we need $R_k$ to be invertible because $G_k$ traces back to \eqref{eq:SCARE-Gc}
which involves $R_k^{-1}$. Indeed $R_k\succ 0$ and thus invertible if $X_k\succeq 0$. This is because if $X_k\succeq 0$, then
$$
R_k =R+\Pi_{22}(X_k)\succeq R\succ 0, \quad
G_k=BR_k^{-1}B^{\T}\succeq 0.
$$

At line 3 of \Cref{alg:SCARE-SDA}
the normalized residual $\NRes(\cdot)$ of \eqref{eq:SCARE-0}, at an approximation
solution $\wtd X$,
\begin{equation}\label{eq:NRes}
\NRes\big(\wtd X\big)
  :=\frac {\|\scrR\big(\wtd X\big)\|_{\F}}
          {2\|A\|_{\F}\,\|\wtd X\|_2+\|Q\|_{\F}+\|\Pi_{11}\big(\wtd X\big)\|_{\F}
                    +\big\|\wtd XB+L_{\rmc}\big(\wtd X\big)\big\|_2^2\big\|\big[R_{\rmc}\big(\wtd X\big)\big]^{-1}\big\|_{\F}},
\end{equation}
is used to measure approximation accuracy and to stop the iterative process if
$\NRes(X_k)\le\epsilon$.
The denominator in \eqref{eq:NRes} is the rough scaling factor such that if $\wtd X$ is rounded from the exact solution
of \SCARE\ \eqref{eq:SCARE-0},  $\NRes\big(\wtd X\big)$ is around the magnitude of the unit machine roundoff.
Here we primarily use the Frobenius norm for computational convenience but we also see the use of the spectral norm.
For the latter, we may use easily computable $\sqrt{\|\wtd X\|_1\|\wtd X\|_{\infty}}$ in place of $\|\wtd X\|_2$, in part because of
$$
\frac 1{\sqrt n}\sqrt{\|\wtd X\|_1\|\wtd X\|_{\infty}}\le\|\wtd X\|_2\le \sqrt{\|\wtd X\|_1\|\wtd X\|_{\infty}},
$$
and similarly for $\big\|\wtd XB+L_{\rmc}\big(\wtd X\big)\big\|_2$.

\begin{algorithm}[t]
\caption{\FP\SDA\ for solving \SCARE\ \eqref{eq:SCARE-0} } \label{alg:SCARE-SDA}
\begin{algorithmic}[1]
    \REQUIRE $A, Q \in \bbR^{n \times n}$, $B, L  \in \bbR^{n \times m}$, $R = R^{\T} \in \bbR^{m \times m}$, $A_0^i \in \bbR^{n \times n}$, $B_0^i \in \bbR^{n \times m}$ for $i = 1, \ldots, r$, and  tolerance $\varepsilon$, and initial approximate $X_0^{\T} = X_0\succeq 0$ if known;
    \ENSURE $X_*$, the last $X_k$, as the computed solution to \eqref{eq:SCARE-0}.
        \hrule\vspace{1ex}
    \STATE $k = 0$;
    \STATE if no initial $X_0$ is provided, set $X_0=0$;
    \WHILE{$\NRes(X_k)>\varepsilon$}
        \STATE $L_k = L + \Pi_{12}(X_k)$,  $R_k = R+\Pi_{22}(X_k)$, $Q_k = Q + \Pi_{11}(X_k)$;
        \STATE $A_k = A - B R_k^{-1} L_k^{\T}$, $G_k = B R_k^{-1} B^{\T}$, $H_k = Q_k - L_k R_k^{-1} L_k^{\T}$;
        \STATE solve \CARE\
               $ 
                     (A_k-G_kX_k)^{\T} Z + Z (A_k-G_kX_k)  - Z G_k  Z + \scrR(X_k)  = 0
               $ 
               for its stabilizing solution  $Z$ by \SDA\ (\Cref{alg:CARE-SDA});
        \STATE $X_{k+1}=X_k+Z$; 
        \STATE $k = k + 1$;
    \ENDWHILE
    \RETURN last $X_k$ as the computed solution.
\end{algorithmic}
\end{algorithm}

\begin{remark}\label{rk:SCARE-SDA}
There are two comments for efficiently implementing \Cref{alg:SCARE-SDA}.
\begin{enumerate}[(1)]
  \item \Cref{alg:SCARE-SDA} is an inter-outer iterative scheme. Its inner iteration is hidden at its line 6 where \Cref{alg:CARE-SDA}
        is called to solve \eqref{eq:FPSCARE(k)} for the difference $X_{k+1}-X_k$. This will result in a more accurate implementation than computing $X_{k+1}$ directly. Here is why.
        As the outer iteration progresses, $X_k$ gradually becomes more and more accurate as an approximation
        to the stabilizing solution of \SCARE\ \eqref{eq:SCARE-0}. Hence $X_{k+1}$ becomes more and more close to $X_k$, and so
        it will be numerically appealing to solve \eqref{eq:FPSCARE(k)} for the difference $Z=X_{k+1}-X_k$. Plugging in $X_{k+1}=X_k+Z$
        to \eqref{eq:FPSCARE(k)} yields the \CARE\ at line 6 there.
  \item The goal of the algorithm is to compute the stabilizing solution of \SCARE\  \eqref{eq:SCARE-0}. The solution to
        \eqref{eq:FPSCARE(k)}, no matter how accurate it is computed, is unlikely to be the one to \eqref{eq:SCARE-0}. Hence
        there is no need to calculate $Z$ at line 6 more accurately than necessary but just enough to update $X_k$, e.g., making
        $X_{k+1}$ a few bits more accurate than $X_k$ as an approximation to the solution of \eqref{eq:SCARE-0}. In general,
        it is hard to know exactly how many more bits accurately is enough. What we do in our current implementation is as follows.
        Denote by $Z_i$ ($i=0,1,\ldots$) the approximations produced by \Cref{alg:CARE-SDA} when it
        is applied to solve $(A_k-G_kX_k)^{\T} Z + Z (A_k-G_kX_k)  - Z G_k  Z + \scrR(X_k)  = 0$ for its stabilizing solution $Z$. We stop the \SDA\ iteration as soon as
        \begin{equation}\label{eq:innerSTOP-FTSDA}
        \|(A_k-G_kX_k)^{\T} Z_i + Z_i (A_k-G_kX_k)  - Z_i G_k  Z_i + \scrR(X_k)\|_{\F}
          \le\tau\cdot\|\scrR(X_k)\|_{\F},
        \end{equation}
        where $0<\tau<1$ is preselected error-reducing factor, as we discussed in \Cref{sec:CARE-SDA}. In our experiments, $\tau=1/8$
        and it works quite well for us.
\end{enumerate}
\end{remark}



\subsection{Convergence Analysis}\label{ssec:conv-FPSDA}
There are a couple of questions about \Cref{alg:SCARE-SDA}  remaining: 1) does \SDA\ at line 6 run without any breakdown? 2) what is its convergence behavior? We will answer these question in this subsection.

We will assume that
\begin{equation}\label{eq:assume-always}
\framebox{
\parbox{9cm}{
Assumptions~\ref{asm:stab} and \ref{asm:dete}, \eqref{eq:quad-func-const},
and \eqref{eq:ker-cond}  hold, and
$(A,B)$ is stabilizable and $(Q - L R^{-1} L^{\T}, A)$ is detectable.
}
}
\end{equation}
Recall that Assumptions~\ref{asm:stab} and \ref{asm:dete} enures that \SCARE\ \eqref{eq:SCARE-0} has a unique \PSD\ solution, which is also
a stabilizing solution. Denote by $X_*$ the \PSD\ solution of \SCARE\ \eqref{eq:SCARE-0}.


Our analysis below will repeatedly call upon the following well-known result, where we have abused notations $A$ and $Q$ as two
generic matrices to state this general fact, whereas everywhere else in this paper, they are tied up with
the targeted \SCARE\ \eqref{eq:SCARE-0} of our focus in this paper. Unlikely, this will cause any confusion.

\begin{lemma}[{\cite{laro:1995}}] \label{lm:LyapnovEq-PSD-prop}
    Let $A,\,Q\in \bbR^{n\times n}$ and suppose that  $A$ is stable, i.e., $\eig(A)\in\bbC_-$, and $Q\preceq 0$. Then
    Lyapunov's equation $A^{\T}X+XA=Q$ has a unique solution $X$ and, moreover, the solution is \PSD.
\end{lemma}


For the sake of analysis, we will assume that $X_{k+1}$ satisfies \eqref{eq:FPSCARE(k)} exactly. Hence our convergence analysis is only indicative as far as the actual behavior of the algorithm in
actual computations is concerned. We point out such an assumption on the inner iteration being exact is not uncommon for analyzing the convergence of
an inner-outer iterative scheme in the literature.

\Cref{thm:SCARE2CARE} below says that \Cref{alg:SCARE-SDA} will be able to generate a sequence $\{X_k\}_{k=0}^{\infty}$, given initial $X_0\succeq 0$. Whether the sequence converges or not is handled in \Cref{{thm:FPSDA-mono-incr},thm:FPSDA-mono-decr} later.

\begin{theorem}\label{thm:SCARE2CARE}
Assume \eqref{eq:assume-always}.
In \Cref{alg:SCARE-SDA} if $X_0\succeq 0$, then the following statements hold.
\begin{enumerate}[{\rm (a)}]
  \item $X_k\succeq 0$ for each $k$;
  \item \CARE\ \eqref{eq:FPSCARE(k)}, for each $k$, has a unique stabilizing solution;
  \item \SDA\ (\Cref{alg:CARE-SDA})  on \CARE\ \eqref{eq:FPSCARE(k)} runs without any breakdowns and is quadratically convergent
        for each $k$.
\end{enumerate}
\end{theorem}

\begin{proof}
We prove that $X_k\succeq 0$ for all $k$ by induction on $k$, and along the way the results in items (b) and (c) are proved as by-products.
Consider $k=0$. By assumption,  $X_0\succeq 0$. Then $(A_0,G_0)$ is stabilizable and $(H_0,A_0)$ is detectable by \Cref{lm:stabilizable,lm:detectable}. Hence,
by \Cref{thm:CARE},
\CARE\ \eqref{eq:FPSCARE(k)} for $k=0$
has a unique stabilizing solution, which is also \PSD.
At convergence, \Cref{alg:CARE-SDA} computes it, which will be denoted by $X_1\succeq 0$ as in the algorithm.
Suppose that $X_k\succeq 0$ for $k=\ell$. Repeat the same argument,
we find that \CARE\ \eqref{eq:FPSCARE(k)} for $k=\ell$ has a unique stabilizing solution, which is also \PSD,
and, at convergence, \Cref{alg:CARE-SDA} computes it too, which is $X_{\ell+1}\succeq 0$ as defined in the algorithm.
\end{proof}





In both \Cref{thm:FPSDA-mono-incr,thm:FPSDA-mono-decr} below, we will show that, with initial $X_0=0$, $X_k$ is monotonically increasing and convergent to $X_{*}$, and, on the other hand,
with initial $X_0\succeq X_*$ such that $\scrR(X_0)\preceq 0$,
$X_k$ is monotonically decreasing and convergent to $X_{*}$.


\begin{theorem} \label{thm:FPSDA-mono-incr}
Assume \eqref{eq:assume-always}.
In \Cref{alg:SCARE-SDA} if $X_0=0$, then the following statements hold.
\begin{enumerate}[{\rm (i)}]
    \item $0=X_0\preceq X_1\preceq\cdots\preceq X_k\preceq X_*$, $\scrR(X_k)\succeq 0$, and $\eig(A_k-G_kX_*)\subset\bbC_-$
    for all $k\ge 0$;
    \item $0\preceq\lim_{k\rightarrow \infty}X_k=X_{*}$ and $\eig\left(A_{\rmc}(X_{*})-G_{\rmc}(X_{*})X_{*}\right)\subset \mathbb{C}_-\bigcup \ibbR$.
\end{enumerate}
\end{theorem}

\begin{proof}
With $X_0=0$, all conclusions of \Cref{thm:SCARE2CARE} are valid.

We  prove, by induction on $k\ge 0$, that
\begin{align}\label{eq:FPSDA-mono-incr:pf-1}
     X_k\preceq X_{k+1},\ X_k\preceq X_*,\ \scrR(X_k)\succeq 0, \text{ and }A_k-G_kX_*\text{ is stable},
\end{align}
which implies item (i).
For $k=0$, we have $0=X_0\preceq X_*$ and $\scrR(X_0)=H_0\succeq 0$. By \Cref{thm:SCARE2CARE}, \CARE\ \eqref{eq:FPSCARE(k)} for $k=0$
has a stabilizing solution, which is also \PSD. As specified by \Cref{alg:SCARE-SDA}, that solution is computed
by \Cref{alg:CARE-SDA} as $X_1\succeq 0=X_0$.
Now, we claim that $A_0-G_0X_*$ is stable. From \Cref{lm:Omega-incr} and using the fact that $X_0\preceq X_*$
and hence $\Omega(X_0)\preceq\Omega(X_*)$ by \Cref{lm:Omega-incr}, we have
\begin{align}\label{eq:FPSDA-mono-incr:pf-2}
(A_0-G_0X_*)^{\T}X_*+X_*(A_0-G_0X_*)
     &=\begin{bmatrix}
          X_*\\
          -I
        \end{bmatrix}^{\T}
        \Omega(X_0)
        \begin{bmatrix}
          X_*\\
          -I
        \end{bmatrix}-H_0-X_*^{\T}G_0X_*\nonumber\\
     &\preceq\begin{bmatrix}
          X_*\\
          -I
        \end{bmatrix}^{\T}
        \Omega(X_*)
        \begin{bmatrix}
          X_*\\
          -I
        \end{bmatrix}-H_0-X_*^{\T}G_0X_*\nonumber\\
     &= -H_0-X_*^{\T}G_0X_*.
\end{align}
Assume, to the contrary, that $A_0-G_0X_*$ is not stable. Then there exist $\bbC^n\ni \by\ne 0$ and $\lambda\in\bbC$ with $\Re(\lambda)\ge 0$ such that
$(A_0-G_0X_*)\by=\lambda \by$. From \eqref{eq:FPSDA-mono-incr:pf-2}, we have
    \begin{align*}
        0\le 2\Re(\lambda)\,\by^{\HH}X_*\by\le -\by^{\HH}H_0\by-\by^{\HH}X_*^{\T}G_0X_*\by.
    \end{align*}
Since $H_0,\, G_0\succeq 0$, we conclude that
$\by^{\HH}H_0\by$=$\by^{\HH}X_*^{\T}G_0X_*\by=0$ and hence
$H_0\by=0$ and $G_0X_*\by=0$. This implies that $A_0\by=\lambda \by$, which is a contradiction because of the detectability of $(H_0,A_0)$
by \Cref{lm:detectable}. Hence, $A_0-G_0X_*$ is stable.
Suppose that \eqref{eq:FPSDA-mono-incr:pf-1} is true for $k=\ell\ge 0$ and next we will show that it holds
for $k=\ell+1$.
We know that $X_{\ell+2}$ is the stabilizing solution of \CARE
\begin{align*}
 \begin{bmatrix} X &  -I \end{bmatrix} \Omega(X_{\ell+1})\begin{bmatrix}
        X\\
     -I  \end{bmatrix}=0,
\end{align*}
implying that $A_{\ell+1} - G_{\ell+1} X_{\ell+2}$ is stable. It can be verified that
\begin{align}
( A_{\ell+1} & - G_{\ell+1} X_{\ell+2})^{\T}(X_{\ell+2}  - X_{\ell+1}) + ( X_{\ell+2} - X_{\ell+1})( A_{\ell+1} - G_{\ell+1} X_{\ell+2})  \nonumber\\
    =& A_{\ell+1}^{\T}X_{\ell+2}+X_{\ell+2}A_{\ell+1}- X_{\ell+2} G_{\ell+1}X_{\ell+2}-(A_{\ell+1}^{\T}X_{\ell+1}+X_{\ell+1}A_{\ell+1}-X_{\ell+1} G_{\ell+1}X_{\ell+1}) \nonumber \\
    &- (X_{\ell+2} - X_{\ell+1}) G_{\ell+1} ( X_{\ell+2} - X_{\ell+1}) \nonumber \\
    =& -\scrR(X_{\ell+1})- (X_{\ell+2} - X_{\ell+1}) G_{\ell+1} ( X_{\ell+2} - X_{\ell+1}). \label{eq:FPSDA-mono-incr:pf-3}
\end{align}
From \Cref{lm:Omega-incr} and using the induction assumption that $X_{\ell}\preceq X_{\ell+1}$, we have
by \Cref{lm:Omega-incr}
\begin{align*}
0=\begin{bmatrix}
        X_{\ell+1}\\
        -I
     \end{bmatrix}^{\T} \Omega(X_{\ell})
    \begin{bmatrix}
        X_{\ell+1}\\
        -I
     \end{bmatrix}
\preceq \begin{bmatrix}
        X_{\ell+1}\\
        -I
     \end{bmatrix}^{\T} \Omega(X_{\ell+1})
     \begin{bmatrix}
        X_{\ell+1}\\
        -I
        \end{bmatrix}
=\scrR(X_{\ell+1}).
\end{align*}
Hence, $X_{\ell+1}\preceq X_{\ell+2}$ by applying \Cref{lm:LyapnovEq-PSD-prop} to \eqref{eq:FPSDA-mono-incr:pf-3}. Using  the induction assumption that $X_{\ell}\preceq X_*$ and \Cref{lm:Omega-incr}, we have
\begin{align*}
(A_{\ell}-G_{\ell}X_*)^{\T}&(X_*-X_{\ell+1})+(X_*-X_{\ell+1})\\
    &=\begin{bmatrix}
        X_*\\
        -I
      \end{bmatrix}^{\T} \Omega(X_{\ell})
      \begin{bmatrix}
        X_*\\
        -I
      \end{bmatrix}-(X_*-X_{\ell+1})G_{\ell}(X_*-X_{\ell+1}) \\
    &\preceq \begin{bmatrix}
        X_*\\
        -I
      \end{bmatrix}^{\T} \Omega(X_*)
      \begin{bmatrix}
        X_*\\
        -I
      \end{bmatrix}-(X_*-X_{\ell+1})G_{\ell}(X_*-X_{\ell+1}) \\
    &= -(X_*-X_{\ell+1})G_{\ell}(X_*-X_{\ell+1}).
\end{align*}
This leads to $X_{\ell+1}\preceq X_*$ because $A_{\ell}-G_{\ell}X_*$ is stable. Now, we claim that $A_{\ell+1}-G_{\ell+1}X_*$ is stable. too.
Assume, to the contrary, that $A_{\ell+1} - G_{\ell+1} X_*$ is not stable. Then there exist $\bbC^n\ni\by \ne 0$ and $\lambda\in\bbC$ with
$\Re(\lambda) \ge 0$ such that $(A_{\ell+1} - G_{\ell+1} X_*) \by = \lambda \by$.
Using $\Omega(X_*)\succeq \Omega(X_{\ell+1})$ by \Cref{lm:Omega-incr} because of $X_{\ell+1}\preceq X_*$ we have just proved, we have
\begin{align*}
    0& =\by^{\HH}\begin{bmatrix}
        X_*\\
        -I
      \end{bmatrix}^{\T} \Omega(X_*)
    \begin{bmatrix}
        X_*\\
     -I  \end{bmatrix}\by\\
    &\ge\by^{\HH}\begin{bmatrix}
        X_*\\
        -I
      \end{bmatrix}^{\T} \Omega(X_{\ell+1})
    \begin{bmatrix}
        X_*\\
     -I  \end{bmatrix}\by\\
     &=\by^{\HH}\left[(A_{\ell+1} - G_{\ell+1} X_*)^{\T}X_*+X_*(A_{\ell+1} - G_{\ell+1} X_*)+X_* G_{\ell+1} X_*+H_{\ell+1}\right]\by\\
     &=2 \Re(\lambda) \by^{\HH}X_*\by+\by^{\HH}X_* G_{\ell+1} X_*\by+\by^{\HH} H_{\ell+1} \by.
\end{align*}
Since $X_*,\, H_{\ell+1},\, G_{\ell+1} \succeq 0$, we find that
$\Re(\lambda)\by^{\HH}X_*\by=\by^{\HH}X_* G_{\ell+1} X_*\by=\by^{\HH} H_{\ell+1} \by=0$, implying $H_{\ell+1} \by = 0$ and $G_{\ell+1} X_* \by = 0$.
This means that
\begin{align*}
    \lambda \by = (A_{\ell+1} - G_{\ell+1} X_*) \by = A_{\ell+1} \by, \quad \by \ne 0, \quad \Re(\lambda) \ge 0, \quad \mbox{ and } \quad H_{\ell+1} \by = 0,
\end{align*}
which contradicts the detectability of $(H_{\ell+1}, A_{\ell+1})$. Therefore, $A_{\ell+1} - G_{\ell+1} X_*$ is stable.
The induction process is completed.

For item (ii), since the sequence $\{X_k\}$ is monotonically increasing and bounded from above by the unique \PSD\  solution $X_*$  of \SCARE\ \eqref{eq:SCARE-0},  the sequence $\{X_k\}$ converges to the \PSD\  solution of \SCARE\ \eqref{eq:SCARE-0}, i.e.,
$\lim_{k\rightarrow \infty}X_k=X_*$. From item (i), all $A_{k} - G_{k} X_{*}$ are stable, by the continuity of matrix eigenvalues with respect to matrix entries, we conclude that $\eig\left(A_{\rmc}(X_{*})-G_{\rmc}(X_{*})X_{*}\right)\subset \mathbb{C}_-\bigcup \ibbR$.
\end{proof}



\begin{theorem} \label{thm:FPSDA-mono-decr}
Assume \eqref{eq:assume-always}.
If
$X_0\succeq X_*$, $A_0-G_0X_0$ is stable, and $\scrR(X_0) \preceq 0$,
then we have the following statements.
\begin{enumerate}[{\rm (i)}]
      \item $X_0\succeq X_1\succeq\cdots\succeq X_k\succeq X_*$ and $\scrR(X_k)\preceq 0$ for each $k\ge 0$;
     \item $\lim_{k\rightarrow \infty}X_k=X_{*}$. 
\end{enumerate}
\end{theorem}

\begin{proof}
With $X_0\succeq X_*\succeq 0$, all conclusions of \Cref{thm:SCARE2CARE} are valid.

We first prove item (i).
We claim that if $X_{k-1}\succeq X_*$, then  $X_k\succeq X_*$. By \CARE\ \eqref{eq:FPSCARE(k)}, we have
     \begin{align}\label{eq:FPSDA-mono-decr:pf-1}
( A_{k-1} - G_{k-1} X_k)^{\T}&(X_* - X_k) + ( X_* - X_k)( A_{k-1} - G_{k-1} X_k)  \nonumber\\
     &=(X_* - X_k) G_{k-1} ( X_* - X_k) + \begin{bmatrix}
        X_*\\
        -I
     \end{bmatrix}^{\T} \Omega(X_{k-1})
     \begin{bmatrix}
        X_*\\
        -I
     \end{bmatrix}.
\end{align}
Since $X_{k-1}\succeq X_*$ and $X_*$ is a solution of \SCARE\ \eqref{eq:SCARE-eq}, it follows from \Cref{lm:Omega-incr} that
\begin{align*}
\begin{bmatrix}
        X_*\\
        -I
     \end{bmatrix}^{\T} \Omega(X_{k-1}) \begin{bmatrix}
        X_*\\
     -I  \end{bmatrix}
\succeq \begin{bmatrix}
        X_*\\
        -I
     \end{bmatrix}^{\T} \Omega(X_*) \begin{bmatrix}
        X_*\\
     -I  \end{bmatrix}=0.
\end{align*}
Since $G_{k-1}\succeq 0$ and $A_{k-1} - G_{k-1} X_k$ is stable,
it follows from \eqref{eq:FPSDA-mono-decr:pf-1} and \Cref{lm:LyapnovEq-PSD-prop} that $X_k\succeq X_*$.
Because $X_0\succeq X_*\succeq 0$, we conclude from what we just proved that $X_k\succeq X_*$ for all $k$.

Next,  we show, by induction, that the sequence $\{X_k\}_{k=0}^{\infty}$ is monotonically decreasing.  First, we show that
$X_0\succeq X_1$. Since $X_1$ is the stabilizing solution of
\begin{align*}
    A_0^{\T}X+XA_0-XG_0X+H_0=0,
\end{align*}
we have
\begin{align*}
    &(A_0-G_0X_0)^{\T}(X_1-X_0)+(X_1-X_0)(A_0-G_0X_0)  
    =(X_0-X_1)G_0(X_0-X_1)-\scrR(X_0)\succeq 0.
\end{align*}
Because $(A_0-G_0X_0)$ is stable, we conclude, by \Cref{lm:LyapnovEq-PSD-prop}, $X_0\succeq X_1$, i.e.,
$X_{k-1}\succeq X_{k}\succeq 0$ holds for $k=0$.
Suppose that $X_{k-1}\succeq X_{k}\succeq 0$ holds for $k=\ell$. We now prove it for $k=\ell+1$.
By \Cref{thm:SCARE2CARE}, $X_{\ell+1}$ is the stabilizing solution of
\begin{align*}
 \begin{bmatrix} X &  -I \end{bmatrix} \Omega(X_{\ell})\begin{bmatrix}
        X\\
     -I  \end{bmatrix}=0.
\end{align*}
For the same reason, $X_{\ell}$ is the stabilizing solution of \eqref{eq:FPSCARE(k)} for $k=\ell-1$ , implying
$A_{\ell-1} - G_{\ell-1} X_{\ell}$ is stable. On the other hand,
\begin{multline}\label{eq:FPSDA-mono-decr:pf-2}
( A_{\ell-1} - G_{\ell-1} X_{\ell})^{\T}(X_{\ell+1} - X_{\ell}) + ( X_{\ell+1} - X_{\ell})( A_{\ell-1} - G_{\ell-1} X_{\ell})  \\
     =  (X_{\ell+1} - X_{\ell}) G_{\ell-1} ( X_{\ell+1} - X_{\ell}) + \begin{bmatrix}
          X_{\ell+1}\\
           -I
        \end{bmatrix}^{\T} \Omega(X_{\ell-1})
        \begin{bmatrix}
          X_{\ell+1}\\
           -I
        \end{bmatrix}.
\end{multline}
By  \Cref{lm:Omega-incr} and using the induction assumption that $X_{\ell-1}\succeq X_{\ell}$, we have
\begin{align*}
\begin{bmatrix}
          X_{\ell+1}\\
           -I
        \end{bmatrix}^{\T} \Omega(X_{\ell-1}) \begin{bmatrix}
        X_{\ell+1}\\
     -I  \end{bmatrix}
\succeq \begin{bmatrix}
          X_{\ell+1}\\
           -I
        \end{bmatrix}^{\T} \Omega(X_{\ell}) \begin{bmatrix}
        X_{\ell+1}\\
     -I  \end{bmatrix}=0.
\end{align*}
Hence, $X_{\ell}\succeq X_{\ell+1}$ by applying \Cref{lm:LyapnovEq-PSD-prop} to \eqref{eq:FPSDA-mono-decr:pf-2}.
This completes the proof of that $\{ X_k \}_{k=0}^{\infty}$ is a monotonically decreasing sequence.
Now that we know  $\{ X_k \}_{k=0}^{\infty}$ is monotonically decreasing, by \Cref{lm:Omega-incr}  we have
\begin{align*}
\scrR(X_k)= \begin{bmatrix}
        X_k\\
        -I
    \end{bmatrix}^{\T} \Omega(X_k)
    \begin{bmatrix}
        X_k\\
        -I
    \end{bmatrix}
\preceq  \begin{bmatrix}
        X_k\\
        -I
    \end{bmatrix}^{\T} \Omega(X_{k-1}) \begin{bmatrix}
        X_k\\
     -I  \end{bmatrix}=0,
\end{align*}
completing the proof of item (i).

Since the \PSD\ solution $X_*$ of \SCARE\ \eqref{eq:SCARE-0} is unique,
item (ii) is a corollary of item (i). 
%
\end{proof}

\section{Newton's  method} \label{sec:Newton}
Often Newton's method is the default when it comes to solving a nonlinear equation. In the current case,
the nonlinear equation is $\scrR(X)=0$.  As such, it comes as
no surprise that Newton's method for solving  \SCARE\ \eqref{eq:SCARE-0} 
has been fully investigated in \cite{dahi:2001} and in \cite{guo:2002a} for a specially case, mostly from the theoretical side in terms
of the iteration updating formula  and convergence analysis.
It is noted that the iteration updating formula is implicitly determined by the Newton step equation, a linear matrix equation known as
the generalized Lyapunov's equation that is numerically difficult even for a modest scale
\cite{bebr:2013,brri:2019,hasi:2021,shss:2016}.


We present Newton's method here chiefly for the purpose of comparing it with our method \FP\SDA\ in \Cref{alg:SCARE-SDA},
and along the way we contribute to its efficient implementation, an issue that was not addressed in \cite{dahi:2001}, while
Guo~\cite{guo:2002a} addressed the issue with one step of a fixed-point type modification to the Newton step equation.
The modification, however, destroyed the quadratic convergence property of Newton's method.

In general, Newton's method needs a sufficiently accurate initial $X_0$
to ensure overall convergence. Finding such an initial is never trivial, if at all possible.
On the other hand, our \FP\SDA\
in \Cref{alg:SCARE-SDA}
is always convergent with $X_0=0$, guaranteed by  \Cref{thm:FPSDA-mono-incr}.

Let $\scrR'_X[E]$ be the Fr\`{e}chet  derivative of $\scrR(\cdot)$ at $X$ along direction $E$:
$$
\scrR'_X[E]=\lim_{t\to 0}\frac 1t\Big(\scrR(X+tE)-\scrR(X)
\Big).
$$
$\scrR'_X$ is a linear operator from $\bbR^{n\times n}$ to itself that maps $E\in\bbR^{n\times n}$ to
$\scrR'_X[E]\in\bbR^{n\times n}$. Formally, Newton's method for solving  \SCARE\ \eqref{eq:SCARE-0} goes as follows:
given initial $X_0\in\bbR^{n\times n}$ such that $X_0\succeq 0$, iterate
\begin{equation}\label{eq:NT4SCARE'}
X_{k+1} = X_k - (\scrR_{X_k}^{\prime})^{-1} [\scrR(X_k)]\, \, \mbox{ for } k = 0, 1, 2, \ldots.
\end{equation}
provided that Fr\`{e}chet  derivatives $\scrR_{X_k}^{\prime}$ as a linear operator are invertible for all $k$.
As before, let
$$
\mbox{$A_k=A_{\rmc}(X_k)$, $R_k = R_{\rmc}(X_k)$, and $L_k=L_{\rmc}(X_k)$}.
$$
Iteration \eqref{eq:NT4SCARE'} is understood as
$\scrR_{X_k}^{\prime}[X_{k+1} - X_k]=-\scrR(X_k)$, yielding
the following nonlinear matrix equation in $X_{k+1}$:
\begin{subequations}\label{eq:NT4SCARE}
\begin{align}
    \what A_k^{\T} X_{k+1} + X_{k+1} \what A_k + \what\Pi_k(X_{k+1}) + M_k = 0, \label{eq:NT4SCARE-step}
\end{align}
where
\begin{align}
S_k &= X_k B+L_k, \label{eq:NT-SXk} \\  
    \what A_k &=  A_k - G_k X_k, \label{eq:NT-AXk} \\
\what\Pi_k(X) &= \begin{bmatrix}
                  I \\
                  -R_k^{-1}S_k^{\T}
                 \end{bmatrix}^{\T} \Pi(X)
                 \begin{bmatrix}
                       I \\
                       - R_{k}^{-1} S_k^{\T}
                 \end{bmatrix}, \label{eq:NT-PiX} \\
M_k &= \begin{bmatrix}
                  I \\
                  -R_k^{-1}S_k^{\T}
                 \end{bmatrix}^{\T}
    \begin{bmatrix}
        Q & L \\ L^{\T} & R
    \end{bmatrix} \begin{bmatrix}
        I \\ - R_{k}^{-1} S_k^{\T}
    \end{bmatrix}. \label{eq:NT-MXk}
\end{align}
\end{subequations}
Iterative formula \eqref{eq:NT4SCARE} has been obtained in \cite{dahi:2001} and in \cite{guo:2002a} for a special case.
Let $T_k=S_kR_k^{-1}$. We get
\begin{align}
\what\Pi_k(X)
   &=\Pi_{11}(X)-\Pi_{12}(X)T_k^{\T}-T_k\big[\Pi_{12}(X)\big]^{\T}+T_k\Pi_{22}(X) T_k^{\T} \nonumber\\
   &=\sum_{i=1}^r\Big({A_0^{i}}^{\T}XA_0^i-{A_0^{i}}^{\T}XB_0^iT_k^{\T}-T_k{B_0^{i}}^{\T}XA_0^i-T_k{B_0^{i}}^{\T}XB_0^iT_k^{\T}\Big) \nonumber\\
   &=\sum_{i=1}^r\big({A_0^{i}}-B_0^{i}T_k^{\T}\big)^{\T}X\big({A_0^{i}}-B_0^{i}T_k^{\T}\big). \label{eq:hatPik}
\end{align}
Plugging this expression into \eqref{eq:NT4SCARE-step}, we find that the equation takes the form of the so-called
the generalized Lyapunov's equation in the literature (see, e.g.,
\cite{bebr:2013,brri:2019,hasi:2021,shss:2016} and references therein).
Numerically, the generalized Lyapunov's equation is hard to deal with even for modest $n$. A popular option is again through the fixed-point idea,
namely freeze $\what\Pi_k(\cdot)$ to yield a Lyapunov's equation, solve the latter, and repeat the process until convergence if the process
converges.

Assuming that \eqref{eq:NT4SCARE-step} is exactly solved, Damm and Hinrichsen \cite{dahi:2001} established the following convergence theorem. Before we state the theorem, we note that
$\what\Pi_0$ defined by \eqref{eq:NT-PiX} and $\scrL_{\what A_0}$ defined as
$\scrL_{\what A_0}(X)=\what A_0^{\T} X + X \what A_0$ for $X\in\bbH^{n\times n}$ are two linear operators on $\bbR^{n\times n}$,
and so is $\scrL_{\what A_0}^{-1}\what\Pi_0$ if $\scrL_{\what A_0}$ is an invertible operator. Notation
$\rho(\scrL_{\what A_0}^{-1}\what\Pi_0)$ is the spectral radius of
$\scrL_{\what A_0}^{-1}\what\Pi_0$ as a  linear operator on $\bbR^{n\times n}$.

\begin{theorem}[{\cite{dahi:2001}}] \label{thm:NT-dahi2001}
Suppose that there exists $X_*\in\bbH^{n\times n}$ such that $\scrR(X_*) \succeq0$.
Given an initial $X_0$, if
\begin{equation}\label{eq:NT-cond-init}
\eig(\what A_0)\subset\bbC_-\quad\mbox{and}\quad
\rho(\scrL_{\what A_0}^{-1}\what\Pi_0)<1,
\end{equation}
then
\begin{enumerate}[{\rm (a)}]
    \item $X_0\succeq X_1\succeq\cdots\succeq X_k \succeq X_*$, $\scrR(X_k) \preceq 0$ for all \tr{$k \geq 1$}, and
    \item $\lim_{k \to \infty} X_k = X_{*+}\succeq X_*$, where $X_{*+}$ is the maximal solution of \eqref{eq:SCARE-0}.
\end{enumerate}
\end{theorem}

Now that we have a linear matrix equation \eqref{eq:NT4SCARE-step} that determines Newton's iterative step and a convergence theorem,
there are two critical issues to be dealt with in order to turn this Newton's method into a practically competitive numerical method
to solve \SCARE\ \eqref{eq:SCARE-0}.

The first issue is how to pick a sufficiently good initial $X_0$ to ensure convergence.
In general Newton's method for nonlinear equations  demands that $X_0$ falls in a sufficiently close proximity of a solution.
There is no exception here. \Cref{thm:NT-dahi2001}
imposes two strong conditions in \eqref{eq:NT-cond-init} on $X_0$ to guarantee convergence.
Firstly, it is highly nontrivial to find a matrix $X_0\in\bbH^{n\times n}$ to satisfy $\eig(\what A_0)\subset\bbC_-$.
Secondly, given that $\bbH^{n\times n}$ is a linear space of dimension
$n(n+1)/2$, seeking a matrix $X_0\in\bbH^{n\times n}$ to satisfy the second condition there is particularly hard, if at all possible.
A possible remedy that we will be adopting is to run \FP\SDA\ (\Cref{alg:SCARE-SDA}) first, which is provably convergent, up
to an approximation $\wtd X$ such that
$$
\NRes\big(\wtd X\big)\le\delta,
$$
before calling Newton's method with $X_0=\wtd X$ just computed by \FP\SDA.
This turns out to be quite an effective strategy for all our numerical examples. Unfortunately, it is not clear how to pick a right $\delta$ that is guaranteed to work, an issue that warrants further investigation.

Even if a proper initial $X_0$ is secured, implementing this Newton's method is not a trivial task because numerically solving
\eqref{eq:NT4SCARE-step} efficiently is non-trivial, as we commented before.
For very small $n$ (a few tens or smaller, for example), it can be turned into a linear system of $n^2$-by-$n^2$ in
the standard form via the Kronecker product
and then solved by the Gaussian elimination. For larger $n$, we will have to resort an iterative solver such as Smith's
method discussed in \Cref{sec:CARE-SDA}.
This is the second issue that we will address in the next  three subsections.
\Cref{alg:NT-SCARE} outlines three variants of Newton's method in one.

\subsection{Direct solution for \eqref{eq:NT4SCARE-step}}
As we just mentioned, for very small $n$, we can simply reformulate \eqref{eq:NT4SCARE-step} via the Kronecker product
as a linear system in the standard form. For ease of presentation, we drop the subscript of $X_{k+1}$
and rewrite the matrix equation as
\begin{equation}\label{eq:NTstep}
\what A_k^{\T} X + X \what A_k + \what\Pi_k(X) + M_k = 0.
\end{equation}

It can be found in most applied matrix theory/computation textbooks that for matrices $C$ and $D$ of apt sizes,
$$
\vex(CXD)=(D^{\T}\otimes C)\vex(X),
$$
where ``$\otimes$'' denotes the Kronecker product, and $\vex(\cdot)$ vectorizes a matrix into a column vector by stacking up the columns of the matrix. By this formula, the linear matrix equation  \eqref{eq:NTstep}
can be readily transformed into a linear system of equations in the standard form
\begin{subequations}\label{eq:NT-Kronecker}
\begin{equation}\label{eq:NT-Kronecker-1}
\scrB_k\vex(X)=-\vex(M_k),
\end{equation}
where
\begin{equation}\label{eq:NT-Kronecker-2}
\scrB_k=I_n\otimes \what A_k^{\T}+\what A_k^{\T}\otimes I_n
    +\sum_{i=1}^r\big({A_0^{i}}-B_0^{i}T_k^{\T}\big)^{\T}\otimes \big({A_0^{i}}-B_0^{i}T_k^{\T}\big)^{\T}\in\bbR^{n^2\times n^2}.
\end{equation}
\end{subequations}
Finally, $X_{k+1}$ is given by
\begin{equation}\label{eq:NTstep'}
\vex(X_{k+1})=\vex(X_k)-\scrB_k^{-1}\vex(\scrB_k\vex(X_k)+\vex(M_k)).
\end{equation}
For all but one of the numerical examples in \Cref{sec:egs}, we have $2\le n\le 9$. This direct method offers a viable option.

\begin{algorithm}[t]
\caption{Variants of \NT\ for solving \SCARE\ \eqref{eq:SCARE-0} } \label{alg:NT-SCARE}
\begin{algorithmic}[1]
\REQUIRE $A, Q \in \bbR^{n \times n}$, $B, L  \in \bbR^{n \times m}$, $R = R^{\T} \in \bbR^{m \times m}$,
         $A_0^i \in \bbR^{n \times n}$, $B_0^i \in \bbR^{n \times m}$ for $i = 1, \ldots, r$,
         initial $X_0^{\T} = X_0$ (sufficiently close to the desired solution for convergence), and  a tolerance $\varepsilon$.
\ENSURE $X_*$, the last $X_k$, as the computed solution to \eqref{eq:SCARE-0}.
        \hrule\vspace{1ex}
\STATE  $k = 0$;
\WHILE{$\NRes(X_k)>\varepsilon$} \label{alg:outer_NT_while}
   \STATE  $S_k = X_k B + L + \Pi_{12}(X_k)$,  $R_k = \Pi_{22}(X_k) + R$;
   \STATE  $\what A_k = A - B R_k^{-1} S_k^{\T}$, $P_k = \begin{bmatrix}
                                                          I \\
                                                          -R_k^{-1}S_k^{\T}
                                                        \end{bmatrix}$,
           $M_k =  P_k^{\T} \begin{bmatrix}
                                     Q & L \\
                                     L^{\T} & R
                                  \end{bmatrix} P_k$;
   \STATE solve \eqref{eq:NT4SCARE-step} directly for $X_{k+1}$ as explained in
          \eqref{eq:NTstep} -- \eqref{eq:NTstep'} and go to line 14;
   \STATE $Y_0 = X_k$, $j=0$; 
   \STATE  $C_{j} = P_k^{\T} \Pi(Y_{j}) P_k + M_k$.
   \WHILE{$\dfrac {\|\what A_k^{\T} Y_j + Y_j \what A_k  + C_{j}\|_{\F}}
                  {2\|\what A_k\|_{\F}\|Y_j\|_2 +\|C_{j}\|_{\F}}
              > \left(\dfrac {\|\what A_k^{\T} Y_0 + Y_0 \what A_k  + C_0\|_{\F}}
                                 {2\|\what A_k\|_{\F}\|Y_0\|_2 +\|C_0\|_{\F}}\right)^2$} 
      \STATE solve \eqref{eq:alg-NTFPSDA} for $Y$, e.g., by Smith's method, a special case of \SDA\
            (see  discussions in \Cref{sec:CARE-SDA} after \Cref{alg:CARE-SDA}), or
             by the Bartels-Stewart method \cite{bast:1972},
              and set $Y_{j+1}$ to be the computed solution;
      \STATE $j=j+1$;
      \STATE  $C_{j} = P_k^{\T} \Pi(Y_{j}) P_k + M_k$;
   \ENDWHILE 
   \STATE $X_{k+1} = Y_{j}$;
   \STATE $k = k + 1$.
\ENDWHILE \label{alg:outer_NT_end}
\RETURN $X_* = X_k$.
\end{algorithmic}
\end{algorithm}

\subsection{Fixed-point iteration for \eqref{eq:NT4SCARE-step}}
Without $\what\Pi_k(X_{k+1})$ or with it freezed at a point instead of being dependent on $X_{k+1}$, \eqref{eq:NT4SCARE-step} is a matrix Lyapunov's equation for which there are a number of direct or iterative methods (see, e.g.,  \cite{bast:1972,belt:2009,simo:2007} and references therein) for $n$ large and small.
Guo~\cite{guo:2002a}, although for a special case, proposed a modified Newton method by simply replacing $\what\Pi_k(X_{k+1})$
with $\what\Pi_k(X_k)$. 
Essentially, Guo's idea is to simply perform one step of the fixed-point iteration on  \eqref{eq:NT4SCARE-step}, but there is no reason
not to do more so that \eqref{eq:NT4SCARE-step} is solved accurately enough to maintain the usual quadratic convergence of
Newton's method.
This is exactly what we will do: with $Y_0=X_0$, solve
\begin{equation}\label{eq:alg-NTFPSDA}
\what A_k^{\T} Y + Y \what A_k  + C_{j}  = 0,
\end{equation}
where $C_{j} = P_k^{\T} \Pi(Y_{j}) P_k + M_k$,
for $Y_{j+1}$ for $j = 0, 1, 2, \ldots$
until the stopping criterion at line 8 of \Cref{alg:NT-SCARE} is met.
The design of the stopping criterion is motivated by the fact that Newton's method is usually quadratically convergent.

At line 9 of \Cref{alg:NT-SCARE}, it is stated to use Smith's method (cf. \eqref{eq:SDA4Lyap})
to solve \eqref{eq:alg-NTFPSDA} iteratively or the Bartels-Stewart method \cite{bast:1972} to solve it directly.
Specifically, we will solve for a correction $Z$ to $Y_j$. Plugging in $Y=Y_j+Z$ to \eqref{eq:alg-NTFPSDA} yields
\begin{equation}\label{eq:alg-NTFPSDA'}
\what A_k^{\T} Z + Z \what A_k  + \underbrace{[\what A_k^{\T} Y_j + Y_j \what A_k+C_{j}]}_{=:\wtd C_j}  = 0.
\end{equation}
There is not much to comment on if solved by the Bartels-Stewart method, which is based on the Schur's decomposition
of $\what A_k$.
According to \eqref{eq:SDA4Lyap}, Smith's method applied to \eqref{eq:alg-NTFPSDA'} is given by
\begin{subequations}\label{eq:alg-NTFPSDA'-smith}
\begin{align}
&S=-(\what A_k+\gamma I)^{\T},\,\,
 E_0=I+2\gamma S^{-\T},\,\,
 Z_0=-2\gamma S^{-1}\wtd C_j S^{-\T},\,\,\mbox{and} \label{eq:alg-NTFPSDA'-smith-1}\\
&E_{i+1} = E_i^2, \, \, Z_{i+1}= Z_i+E_i^{\T}Z_iE_i\,\,\,\mbox{for $i\ge 0$}, \label{eq:alg-NTFPSDA'-smith-2}
\end{align}
\end{subequations}
and, finally, $Y_{j+1}=Y_j+Z_i$ where $Z_i$ is the last approximation at convergence, based on the stopping criterion
$$
\|\what A_k^{\T} Z_i + Z_i \what A_k+\wtd C_j\|_{\F}\le\tau\cdot\|\wtd C_j\|_{\F}
$$
for a suitable $\tau$, say, $1/8$. This adopts a similar strategy as in \eqref{eq:innerSTOP-FTSDA}, with the purpose
that $Z_i$ has certain degree of accuracy to improve $Y_j$. Ideally, $Z_i$ is made to have just enough accuracy so that
any more accuracy than it already has as an approximate solution  to  \eqref{eq:alg-NTFPSDA'} won't help $Y_{j+1}$ as
an approximate solution to \eqref{eq:NT4SCARE-step}.


Again \eqref{eq:alg-NTFPSDA'-smith} involves a parameter $\gamma$ to be chosen. Optimal $\gamma$ is determined by
$$
\gamma:=\arg\min_{\lambda\in\eig(\what A_k)}\left|\frac {\lambda-\gamma}{\lambda+\gamma}\right|.
$$
For \SCARE\ \eqref{eq:SCARE-0} of interest, if $X_k$ is sufficiently accurate, then $\eig(\what A_k)\subset\bbC_-$ and hence
optimal $\gamma<0$. As discussed in \Cref{sec:CARE-SDA}, in our later experiments, we determine a suboptimal
$\gamma$ by encircling $\eig(\what A_k)$ with a rectangle $[a,b]\times [-c,c]$ where $a\le b<0$, and then setting
$\gamma$ as in \eqref{eq:gamma-suboptimal}.

Putting all together, we arrive at three variants of Newton's method: one plain Newton's method with each
Newton step solved directly and two inner-outer iterative schemes of two or three levels,
as outlined in \Cref{alg:NT-SCARE}. The latter two schemes are Newton's method combined with the fixed-point strategy, for solving \SCARE\ \eqref{eq:SCARE-0},
where the first level, the outer iteration, is the Newton iteration, the second level is
the fixed-point iteration to calculate the Newton iterative step  determined by \eqref{eq:NT4SCARE-step}.
If an iterative scheme such as Smith's method, a special case of \SDA, is used to calculate each fixed-point iterative step \eqref{eq:alg-NTFPSDA}, via
\eqref{eq:alg-NTFPSDA'} and \eqref{eq:alg-NTFPSDA'-smith}, that will be the third level.


%
%


The stopping criterion at line 8 of \Cref{alg:NT-SCARE} is designed with the purpose to capture the locally quadratic convergence
of Newton's method. Increasingly, the number of the fixed-point iterations for
calculating the next Newton's approximation determined by \eqref{eq:NT4SCARE-step} is likely increases as $X_k$ becomes more
and more accurate. Another option is to simply solve \eqref{eq:NT4SCARE-step} with one fixed-point step
as in \cite{guo:2002a}.
This yields
a modified Newton's method as is outlined in \Cref{alg:mNT-SCARE}.

\section{Numerical experiments} \label{sec:egs}
In this section, we will test and compare two methods that we outlined in \Cref{sec:intro}:
\begin{enumerate}[(1)]
  \item \FP\SDA\ in \Cref{alg:SCARE-SDA}: Iteratively freeze the coefficient matrices
        $A_{\rmc}(\cdot)$, $G_{\rmc}(\cdot)$, and $H_{\rmc}(\cdot)$ in \eqref{eq:SCARE-CARE-1} and
        solve the resulting \CARE\ by the doubling algorithm (\SDA). With $X_0=0$ initially, the method always monotonically converges by \Cref{thm:FPSDA-mono-incr}.
  \item  \NT\ in \Cref{alg:NT-SCARE}: It includes three variants of Newton's method. To distinguish them, we will use
         $\NT_1$ for Newton's step equation \eqref{eq:NT4SCARE-step} solved by transforming the equation to
         a linear system in standard form \eqref{eq:NT-Kronecker}, and both $\NT_2$ and $\NT_3$ for Newton's step equation solved iteratively
         with associated Lyapunov equations by MATLAB's {\tt lyap} in  $\NT_2$
         and by Smith's method in $\NT_3$, respectively.

         In all tests for these \NT\ variants,  \FP\SDA\ is called first to calculate an initial $X_0$
         such that
         \begin{equation}\label{eq:NT-initial-to-get}
         \NRes(X_0)\le\delta,
         \end{equation}
         where $\delta$ is some small number to be specified.
\end{enumerate}
These methods will be tested on eight examples of \SCARE\ \eqref{eq:SCARE-0}, obtained from
modifying some of the \CARE\ examples -- real or artificial --
in the literature by adding stochastic components in
$A_0^1, \ldots, A_0^r$, $B_0^1, \ldots, B_0^r$  constructed by ourselves.
The first four examples are made from artificial \CARE\ that we will use to validate our algorithms and
implementations, while
the other four examples are made from real-world applications and they are harder.

All computations are performed in MATLAB 2023a.
The stopping criterion for the (outmost) iteration is $\NRes(X_k)\le 10^{-14}$.

\subsection{Test problems for validation}\label{ssec:easy-examples}
The four examples here are  small artificial \CARE\ examples in the literature with newly added
stochastic components in $A_0^1, \ldots, A_0^r$, $B_0^1, \ldots, B_0^r$.

\begin{figure}[t]
\center
\begin{subfigure}[b]{0.42\textwidth}
\center
\includegraphics[width=\textwidth]{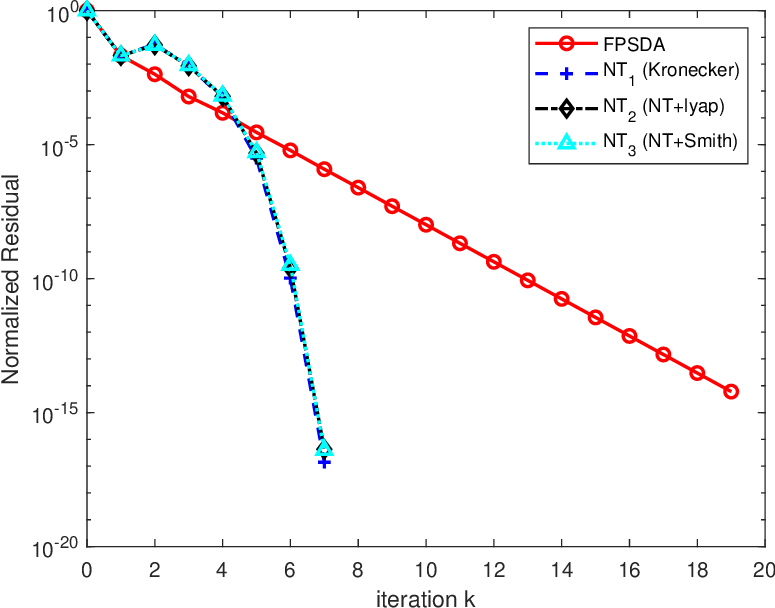}
\caption{\Cref{example:1}}
\label{fig:example1_convergence}
\end{subfigure}\quad
\begin{subfigure}[b]{0.42\textwidth}
\center
\includegraphics[width=\textwidth]{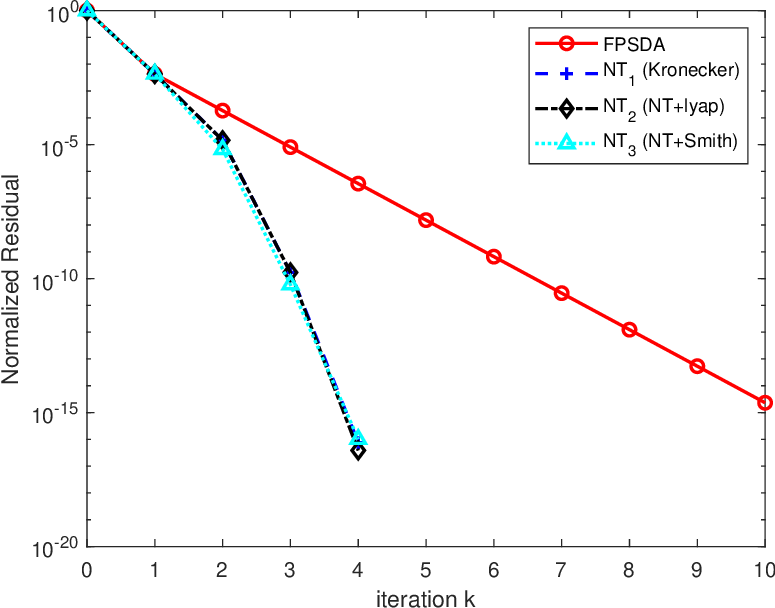}
\caption{\Cref{example:12} with $\varepsilon = 0.01$}
\label{fig:example12_convergence}
\end{subfigure}
\begin{subfigure}[b]{0.42\textwidth}
\center
\includegraphics[width=\textwidth]{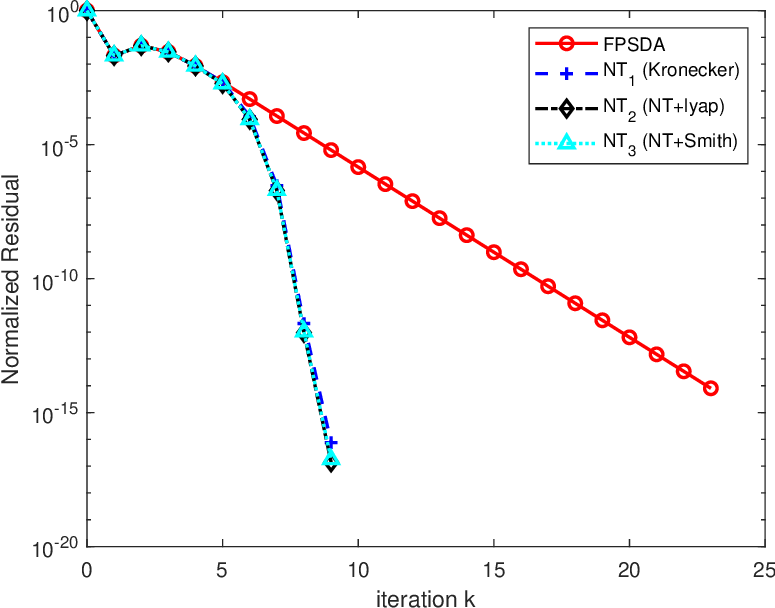}
\caption{\Cref{example:Weng}}
\label{fig:example_Weng}
\end{subfigure}\quad
\begin{subfigure}[b]{0.42\textwidth}
\center
\includegraphics[width=\textwidth]{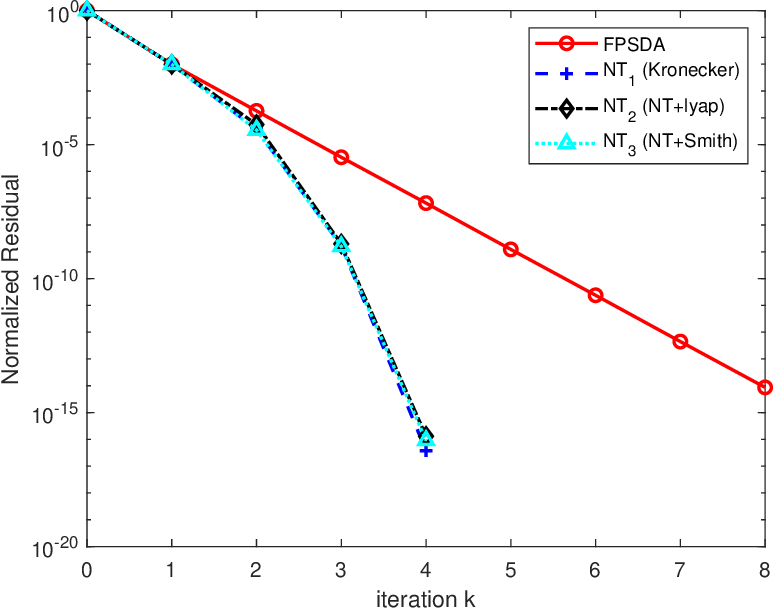}
\caption{\Cref{example:rho_NT_GT_1} with $\varepsilon = 5$}
\label{fig:rho_NT_GT_1}
\end{subfigure}
\caption{
Iterative histories of \FP\SDA, $\NT_1$ (via Kronecker's reformulation of each Newton step equation),
$\NT_2$ (via MATLAB's {\tt lyap} to iteratively solve Newton step equations),
and $\NT_3$ (via Smith's method to iteratively solve Newton step equations in an inner-outer fashion).
The plotted histories of normalized residual by the variants of Newton's method include the portion by \FP\SDA\
for calculating an initial.
}
\label{fig:easy-Egs}
\end{figure}

\begin{example} \label{example:1}
$A$, $B$, $Q$ and $R$ are taken from  \cite[Example 2.2]{abbe:1999b}:
$$
     A = \begin{bmatrix}
         0.9512 & 0 \\  0 & 0.9048
     \end{bmatrix}, \quad
     B = \begin{bmatrix}
         4.8770& 4.8770 \\ -1.1895 & 3.5690
     \end{bmatrix}, \quad Q = \begin{bmatrix}
         0.005 & 0 \\ 0 & 0.020
     \end{bmatrix}, \quad R = \begin{bmatrix}
         \frac{1}{3} & 0 \\  0 & 3
     \end{bmatrix},
$$
while we set $L=0$, $r=3$ and
\begin{align*}
A_0^1 &= \begin{bmatrix}
         -0.1 & 0.1 \\  -0.2 & 0.2
     \end{bmatrix}, \quad A_0^2 = \begin{bmatrix}
         1 &  -0.1 \\ 0.5 & 0
     \end{bmatrix}, \quad A_0^3 = \begin{bmatrix}
         0 & -0.2 \\ 0.2 & 0.5
     \end{bmatrix}, \\
B_0^1 &= \begin{bmatrix}
         0 & -0.1 \\  0.1 & 0
     \end{bmatrix}, \quad B_0^2 = \begin{bmatrix}
         0.5 & 1 \\ -0.1 & 0.2
     \end{bmatrix}, \quad B_0^3 = \begin{bmatrix}
         1 & -1 \\ -0.2 & 1
     \end{bmatrix}.
\end{align*}
\end{example}

\begin{example} \label{example:12}
$A$, $B$, $Q$ and $R$ are taken from \cite[Example 12]{chfl:2005}:
\begin{align*}
&A = \varepsilon \begin{bmatrix}
         \frac{7}{3} & \frac{2}{3} & 0 \\
         \frac{2}{3} & 2 & - \frac{2}{3} \\
         0 & -\frac{2}{3} & \frac{5}{3}
    \end{bmatrix}, \quad
    Q = \begin{bmatrix}
         (4\varepsilon+4+\varepsilon^{-1})/9 &  2(2\varepsilon-1-\varepsilon^{-1})/9 & 2(2-\varepsilon-\varepsilon^{-1})/9 \\
         2(2\varepsilon-1-\varepsilon^{-1})/9 & (1+4\varepsilon+4/\varepsilon)/9  &    2(-1-\varepsilon+2/\varepsilon)/9 \\
                      2(2-\varepsilon-\varepsilon^{-1})/9 &  2(-1-\varepsilon+2/\varepsilon)/9  &   (4+\varepsilon+4/\varepsilon)/9
    \end{bmatrix}, \\
&B = \frac{1}{\sqrt{\varepsilon}} I_3, \quad R = I_3,
\end{align*}
where $\varepsilon = 0.01$, while we set $L=0$, $r=1$ and
$$
A_0^1 =   0.1  \begin{bmatrix}
         0.1 & -0.1 & 0.01 \\  -0.2 & 0.1 & -0.1 \\  0.05 & -0.01 & 0.3
    \end{bmatrix},  \quad
\quad B_0^1 = 0.1 \begin{bmatrix}
         0 & 0 & 0.2 \\ 0.36 &  -0.6 & 0 \\  0 & -0.95 & -0.032
    \end{bmatrix}.
$$
\end{example}

\begin{example} \label{example:Weng}
$A$, $B$, $Q$, $L$ and $R$ are taken from \cite{weph:2021}:
\begin{align*}
        A &= \begin{bmatrix} 0.9512 & 0 \\  0 & 0.9048 \end{bmatrix}, \quad
        B = \begin{bmatrix} 4.8770 & 4.8770 \\ -1.1895& 3.5690 \end{bmatrix}, \quad Q = \begin{bmatrix} 0.0028 & -0.0013 \\ -0.0013 & 0.0190 \end{bmatrix}, \\
        R &= \begin{bmatrix} 1/3 & 0 \\ 0 & 3\end{bmatrix}, \quad L = 0,
\end{align*}
while $r=1$ and we modify $A_0^1$ and $B_0^1$ in \cite{weph:2021} to
$$
A_0^1  = 6.5 \begin{bmatrix} 0.1 & 0.2 \\ 0.2 & 0.1 \end{bmatrix}, \quad B_0^1 = 6.5 I_2.
$$
\end{example}

\begin{example} \label{example:rho_NT_GT_1}
$A$, $B$, $Q$ and $R$ are taken from \cite[Example 11]{chfl:2005}:
$$
A = \begin{bmatrix} 3- \varepsilon & 1 \\ 4 & 2-\varepsilon \end{bmatrix}, \quad
B = \begin{bmatrix} 1 \\ 1 \end{bmatrix}, \quad
Q = \begin{bmatrix}
        4\varepsilon-11 & 2\varepsilon-5 \\
        2\varepsilon-5 & 2\varepsilon-2
    \end{bmatrix}, \quad
R = 1,
$$
where $\varepsilon = 5$, while  we set $L=0$, $r=1$ and
$$
A_0^1  = \begin{bmatrix} 0.1 & -0.1 \\ -0.2 & 0.1 \end{bmatrix}, \quad
B_0^1  =  \begin{bmatrix} 0.1 \\ 0 \end{bmatrix}.
$$
\end{example}

\setlength{\tabcolsep}{.2em}
\begin{table}[t]
  \centering
  \begin{tabular}{|l|c|c|c|c|c|c|c|c|c|} \hline
    &  & \multicolumn{2}{c|}{$\NT_1$} & \multicolumn{2}{c|}{$\NT_2$ (\NT +lyap)} & \multicolumn{2}{c|}{$\NT_3$ (\NT +smith)} &  \multicolumn{2}{c|}{init. (\FP\SDA)} \\ \cline{3-10}
    & \FP\SDA & \#itn. & sol. err. &  \#itn. & sol. err. & \#itn. & sol. err. & $\delta$ {\scriptsize in \eqref{eq:NT-initial-to-get}} & \#itn.  \\ \hline
\Cref{example:1}    & (19,21) &  6 & $3.7{\scriptstyle (-14)}$  & (6,23) & $3.7{\scriptstyle (-14)}$  & (6,28,28) & $3.7{\scriptstyle (-14)}$  & $0.5$ & (1,2) \\
\Cref{example:12}   & (10,41) &  3 & $4.3{\scriptstyle (-15)}$ & (3,8)  & $4.4{\scriptstyle (-15)}$ & (3,11,11) & $9.4{\scriptstyle (-15)}$ & $0.5$ & (1,5) \\
\Cref{example:Weng} & (23,24) &  5 & $1.8{\scriptstyle (-13)}$  & (5,30) & $1.9{\scriptstyle (-13)}$ & (5,30,30) & $1.8{\scriptstyle (-13)}$ & $10^{-2}$ & (4,5) \\
\Cref{example:rho_NT_GT_1} & (8,8) & 3 & $2.0{\scriptstyle (-14)}$ & (3,8) & $2.1{\scriptstyle (-14)}$ & (3,10,10) & $2.0{\scriptstyle (-14)}$ & $0.5$ & (1,1) \\ \hline
  \end{tabular}
  \caption{The numbers of outer and inner iterations  and solution error against the one by \FP\SDA, for solving \SCARE\ \eqref{eq:SCARE-0}.
  For an inner iteration, it is the total number,
  e.g., the total number of \SDA\ iterations for \Cref{example:1} by \FP\SDA\ is 21. The iteration scheme of $\NT_3$ has 3 levels and
  hence each run of $\NT_3$ produces three numbers, e.g., for \Cref{example:1}, $\NT_3$ consumes 6 Newton steps, a total of 28 Lypunov equations
  solved by Smith's method and, altogether, a total of 28 Smith iterations are executed. The ``sol. err.'' is calculated
  as in \eqref{eq:solerr} and listed in such a way that, e.g., $3.7{\scriptstyle (-14)}$ represents $3.7\cdot 10^{-14}$.}
  \label{tab:perf-stat}
\end{table}

\Cref{fig:easy-Egs} plots the iterative histories in terms of normalized residuals
by \FP\SDA\ and three variants of Newton's methods.
The initials for three variants of Newton's method are calculated by \FP\SDA\ to satisfy \eqref{eq:NT-initial-to-get} with $\delta$
as specified in \Cref{tab:perf-stat}. It is noted that
$\delta=10^{-2}$ is used for \Cref{example:Weng}, much smaller than $\delta=0.5$ for the other three examples. It turns out that
$\NT_2$ and $\NT_3$ diverge if with  $\delta=0.5$ while
$\NT_1$ still converges with $\delta=0.5$. This suggests that using the fixed-point type scheme for solving Newton step equations may need
more accurate initial than the plain Newton's method.
\Cref{tab:perf-stat} contains performance statistics of the numbers of outer iterations and the total numbers of inner iterations (in each level), and relative solution errors:
\begin{equation}\label{eq:solerr}
\frac {\|X_{\NT}-X_{\FP\SDA}\|_{\F}}{\|X_{\FP\SDA}\|_{\F}}
\end{equation}
for the solutions $X_{\NT}$ by any of the three Newton variants against $X_{\FP\SDA}$ by \FP\SDA. The last column in \Cref{tab:perf-stat}
lists the numbers of outer and inner iterations by \FP\SDA\ to calculate initials for Newton's method to use.

From \Cref{fig:easy-Egs} and \Cref{tab:perf-stat}, we made the following observations:
\begin{enumerate}[(i)]
  \item \FP\SDA\ converges linearly while Newton's method converges quadratically.
  \item On average, there are only a few number of inner-most iterations per equation solving. In the case of \FP\SDA, for all but \Cref{example:12}, each   freezed \CARE\ takes about one \SDA\ iteration which means the approximation $X_0$ at line 3 of \Cref{alg:CARE-SDA} suffices, while for \Cref{example:12}, the average number of \SDA\ iterations per \CARE\ is about $4$.
  \item The stopping criterion at line 8 of \Cref{alg:NT-SCARE} can capture the quadratic convergence of Newton's method at
        a cost of a few number of fixed-point iterations to solve each Newton step equation.
\end{enumerate}

\subsection{Real-World Problems}\label{ssec:hard-examples}

For the next four examples, $A_0^1, \ldots, A_0^r$ and $B_0^1, \ldots, B_0^r$
are constructed randomly with MATLAB's {\tt wgn} function for white Gaussian noises. Hence, they provides infinitely many testing problems.
They are considerably harder than the previous four problems. In particular, \Cref{example:6} is too
big ($n=199$) for $\NT_1$ in which Newton step equations are transformed to linear systems in the standard form and
then solved by Gaussian elimination. More detailed observations will come later.
The initials to the variants of Newton's method are again calculated by \FP\SDA\ such that
\eqref{eq:NT-initial-to-get} is satisfied with $\delta$
as specified in \Cref{tab:perf-stat:hard-examples}.

\begin{example} \label{example:6}
The coefficient matrices, $A$, $B$, $Q$, $R$, and $L$ are taken from a control model of position and velocity for a string of high-speed vehicles \cite{abbe:1999a}:
\begin{align*}
A &= \begin{bmatrix}
          E & F \\ & \ddots & \ddots \\
          & & E & F \\
          & & & E & -1 \\
          & & & & -1
     \end{bmatrix} \in \bbR^{(2m-1) \times (2m-1)} \quad\mbox{with}\quad E = \begin{bmatrix}
         -1 & 0 \\ 1 & 0
     \end{bmatrix}, \quad F = \begin{bmatrix}
         0 & 0 \\ -1 & 0
     \end{bmatrix}, \\
B &= [b_{ij}] \in \bbR^{(2m-1) \times m}\quad\mbox{with}\quad b_{ij} = \begin{cases}
        1,  & \mbox{ for } i = 1, 3, 5, \ldots, 2m-1, j = (i+1)/2,\\
        0, & \mbox{ otherwise},
        \end{cases} \\
Q &= [q_{ij} ]\in \bbR^{(2m-1) \times (2m-1)}\quad\mbox{with}\quad q_{ij} = \begin{cases}
     10, & \mbox{ for } i = 2, 4, 6, \ldots, 2m-1, \ j = i, \\
     0, & \mbox{ otherwise},
     \end{cases} \\
     R &= I_m, \quad L = 0_{(2m-1)\times m},
\end{align*}
where $m$ is the number of vehicles, and hence $n=2m-1$, and
$A_0^1, \ldots, A_0^5$ and $B_0^1, \ldots, B_0^5$ are multiplicative white noises constructed as
\begin{align*}
    A_0^i &= 0.1 \times i \times \frac{\| A \|_{\infty}}{\| \widehat{A}_0^i \|_{\infty}} \widehat{A}_0^i \text{ with }\widehat{A}_0^i = \mbox{\texttt{wgn}}(2m-1,2m-1,8\times i)\,\,   \mbox{ for } i = 1, \ldots, 5, \\
     B_0^i &= 0.15 \times i \times \frac{\| B \|_{\infty}}{\| \widehat{B}_0^i \|_{\infty}} \widehat{B}_0^i\text{ with }\widehat{B}_0^i = \mbox{\texttt{wgn}}(2m-1,m,3\times i)\,\,   \mbox{ for } i = 1, \ldots, 5.
\end{align*}
In our test, $m=100$ yielding $n=199$, which turns out to be too big for $\NT_1$.
\end{example}

\begin{example} \label{example:3}
In \cite{lwhxwl:2023,nakm:2021a}, a 3D missile/target interception engagement with impact angle guidance strategies
is modeled by a  state-dependent
linear-structured dynamic system in the form of \eqref{eq:DSys-SSDC}, where \SDC\ matrices and
white noise matrices are given by
\begin{subequations} \label{eq6.1}
\begin{align}
    A(\bx)&=
    \begin{bmatrix}
    0 & 1 &0 & 0 & 0
    \\ 0 &\frac{-2\dot{\varrho}}{\varrho} & 0 & \frac{-1}{2}\dot{\psi}\sin(2x_1+2\theta_{f}) & \frac{g_1}{z_{a}}
    \\ 0 & 0 & 0 & 1 & 0
    \\0 &2\dot{\psi}\tan(x_1+\theta_{f}) &0 &\frac{-2\dot{\varrho}}{\varrho} &\frac{g_{2}}{z_{a}}
    \\0 &0 &0 &0 &-\eta
    \end{bmatrix}, \quad
    B(\bx)=
    \begin{bmatrix}
    0&0\\
    \frac{-\cos\theta_{M}}{\varrho} &0 \\
    0&0\\
    \frac{\sin\theta_{M}\sin\psi_{M}}{\varrho\cos\theta} &\frac{-\cos\psi_{M}}{\varrho\cos\theta}\\
        0 &0
    \end{bmatrix},\nonumber\\
     A_0^i &= 0.2 \times i \times  \frac{\| A \|_{\infty}}{\| \widehat{A}_0^i \|_{\infty}} \widehat{A}_0^i \ \text{ with }\ \widehat{A}_0^i = \mbox{\texttt{wgn}}(5,5,8\times i)\,\, \mbox{ for } i = 1, \ldots, 4, \label{eq6.1a}\\
     B_0^i &= 0.1 \times i \times  \frac{\| B \|_{\infty}}{\| \widehat{B}_0^i \|_{\infty}} \widehat{B}_0^i \ \text{ with } \ \widehat{B}_0^i = \mbox{\texttt{wgn}}(5,2,3\times i)\,\, \mbox{ for } i = 1, \ldots, 4, \label{eq6.1b}
\end{align}
\end{subequations}
in which $\varrho$ measures the distance between the missile and the target, $\{ \psi, \theta\}$ (respective to $\{ \psi_{M}, \theta_M\}$) are the azimuth and elevation angles (respective to missile) corresponding to the initial frame (respective to the line-of-sight (\LOS) frame).
Furthermore, the state vector $\bx = [x_1, x_2, x_3, x_4, x_5]^{\T}$ satisfies
$$
x_1 = \theta - \theta_f,\,\,
x_2 = \dot{x}_1,\,\,
x_3 = \psi - \psi_f,\,\,
x_4 = \dot{x}_3
$$
with $\theta_f$ and $\psi_f$ being the prescribed final angles, $x_5$ is a slow varying stable auxiliary
variable governed by $\dot{x}_5 = -\eta x_5$ for some $\eta > 0$, $g_1$ and $g_2$ are highly nonlinear
functions of $\varrho$, $\theta$, $\psi$, $\theta_f$, $\psi_f$, $\theta_T$, $\psi_T$, $a_T^z$ and $a_T^y$
(see \cite{nakm:2021a} for details), where $\{\psi_T, \theta_T\}$ are the azimuth and elevation angles of
the target to the \LOS\ frame, $a_T^z$ and $a_T^y$ are the lateral accelerations for the target,
and $(u_1, u_2) = (a_M^z, a_M^y)$ is the control vector for the maneuverability of missile.

As we explained in \Cref{sec:intro}, the \SSDRE\ approach will repeatedly solve stochastic
state-dependent \CARE\
\eqref{eq:SSDCARE} at a number of fixed state $\bx$. In this test, we will consider one such equation
as an \SCARE\ \eqref{eq:SCARE-0} with
\begin{align*}
     A &= \begin{bmatrix}
          0 & 1          & 0 & 0           & 0 \\
          0 & 0.0696 & 0 & -0.0307 & -1.91 \times 10^{-4} \\
          0 & 0          & 0 & 1           & 0 \\
          0 & 0.123   & 0 & 0.0696 & 6.13 \times 10^{-4} \\
          0 & 0         & 0 & 0           & -0.1
     \end{bmatrix}, \quad B = \begin{bmatrix}
          0 & 0 \\ -9.13 \times 10^{-5} & 0 \\ 0 & 0 \\ 2.42 \times 10^{-5} & -1.30 \times 10^{-4} \\ 0 & 0
     \end{bmatrix}, \\
     Q &= {\rm diag}(1000, 1000, 1000, 1000, 0),  \quad R = I_2, \quad L = 0, \quad A_0^i \text{ and }B_0^i\text{ as in \eqref{eq6.1}.}
\end{align*}
\end{example}

\begin{example} \label{example:F16}
The F16 aircraft flight control system \cite{cpws:2022} can be described as
a linear-structured dynamic system in the form of \eqref{eq:DSys-SSDC}:
\begin{align*}
      \dot{\bx}\equiv \begin{bmatrix}\dot{u}\\
      \dot{v}\\
      \dot{w}\\
      \dot{p}\\
      \dot{q}\\
      \dot{\varrho}\end{bmatrix} = & \begin{bmatrix}
           (g \sin\theta)/u & 0 & 0 & 0 & -w & v \\
           (-g\sin \phi \cos \theta)/u & 0 & 0 & w & 0 & -u \\
           (-g\cos \phi \cos \theta)/u & 0 & 0 & -v & u & 0 \\
           0 & 0 & 0 & c_1q/2 & (c_1p+c_2\varrho)/2 & c_2 q/2 \\
           0 & 0 & 0 & c_3 p + c_4\varrho/2 & 0 & c_4p/2 - c_3 \varrho \\
           0 & 0 & 0 & c_5q/2 & (c_5p+c_6\varrho)/2 & c_6 q/2
      \end{bmatrix}\begin{bmatrix}u\\
      v\\
      w\\
      p\\
      q\\
      \varrho\end{bmatrix}\\ &+\begin{bmatrix}
          1/\mu & 0 & 0 & 0 \\
          0 & 0 & 0 & 0 \\
          0 & 0 & 0 & 0 \\
          0 & c_{lp} & 0 & c_{\nu p} \\
          c_{\mu q} Z_{TP} & 0 & c_{\mu q} & 0 \\
          0 & c_{l\varrho} & 0 & c_{\nu\varrho}
      \end{bmatrix} \begin{bmatrix}F_T\\
      L\\
      M\\
      N\end{bmatrix}
      \equiv A(\bx) \bx + B(\bx) \bu,
\end{align*}
where $g$ is the gravity force, $\mu$ is the aircraft
mass, $(u, v, w)$ and $(p, q, \varrho)$ are the aircraft velocity and angular velocity vectors, respectively, for roll $\phi$, pitch $\theta$ and yaw $\psi$ angles, parameters $c$ with various subscripts are the suitable combinations of coefficients of the aircraft inertial matrix, the control vector $\bu=[F_T,L,M,N]^{\T}$ consists of the thrust $F_T$ and the aircraft moment vector $[L, M, N]^{\T}$, and $Z_{TP}$ is the position of the thrust point.

In this test, we consider \SCARE\ \eqref{eq:SCARE-0}, from solving the optimization problem \eqref{eq:quad-func}
at a fixed state by the \SSDRE\ approach, with coefficient matrices:
%
    \begin{align*}
        A &= \begin{bmatrix}
            3.958 \times 10^{-5}  &  0 &  0 & 0 & -5.866 & -6.985 \\
            2.116\times 10^{-4}  & 0 & 0 & 5.866  &  0  &  -84.66 \\
            -0.1158  & 0  & 0 & 6.985  & 84.66  &  0 \\
            0 & 0 & 0 & 1.791\times 10^{-4} & 4.303 \times 10^{-3}  &  -5.006 \times 10^{-3} \\
            0 & 0 & 0 & -5.329 \times 10^{-3} &  0  &  -4.259 \times 10^{-2} \\
            0 & 0 & 0 & -4.769 \times 10^{-3} &  3.253 \times 10^{-2} & -1.791 \times 10^{-4}
        \end{bmatrix},  \\ 
        B &= \begin{bmatrix}
            1.076 \times 10^{-4} & 0 & 0 & 0 \\
            0 & 0 & 0 & 0 \\
            0 & 0 & 0 & 0 \\
            0 & 7.780 \times 10^{-5} & 0 & 7.780 \times 10^{-5} \\
            3.964 \times 10^{-6} &  0  &   1.321 \times 10^{-5} & 0 \\
            0 & 1.211 \times 10^{-6} & 0 & 1.171 \times 10^{-5}
        \end{bmatrix},\\ 
        Q &= 5000 I_6, \quad R = 2 \times 10^{-4} I_4, \quad L = 0,\\
        A_0^i &= 0.012 \times i \times  \frac{\| A \|_{\infty}}{\| \widehat{A}_0^i \|_{\infty}} \widehat{A}_0^i \ \text{ with }\ \widehat{A}_0^i = \mbox{\texttt{wgn}}(6,6,100 \times i)\,\, \mbox{ for } i = 1, 2, 3, \\
        B_0^i &= 0.012 \times i \times  \frac{\| B \|_{\infty}}{\| \widehat{B}_0^i \|_{\infty}} \widehat{B}_0^i \ \text{ with }\ \widehat{B}_0^i = \mbox{\texttt{wgn}}(6,4,40 \times i)\,\, \mbox{ for } i = 1, 2, 3.
    \end{align*}
\end{example}

\begin{figure}[t]
\center
\begin{subfigure}[b]{0.42\textwidth}
\center
\includegraphics[width=\textwidth]{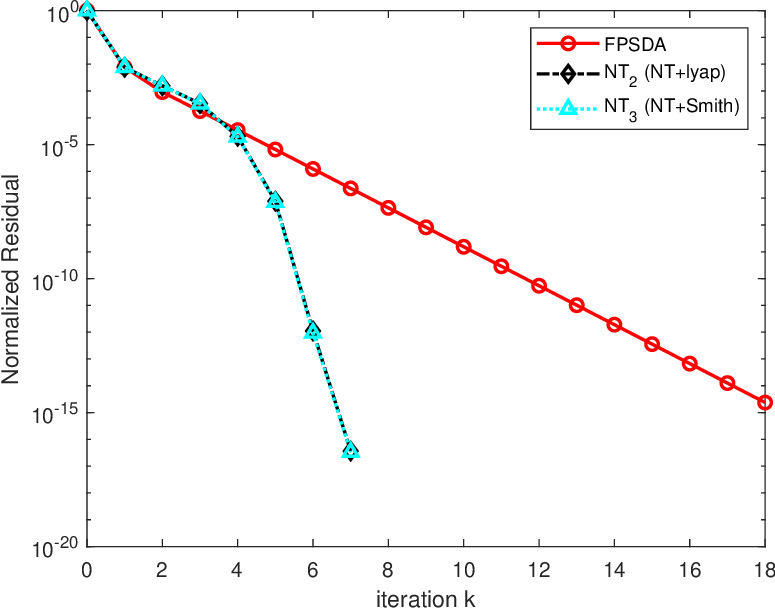}
\caption{Example~\ref{example:6} with $m = 100$}
\label{fig:example6__convergence}
\end{subfigure}\quad
\begin{subfigure}[b]{0.42\textwidth}
\center
\includegraphics[width=\textwidth]{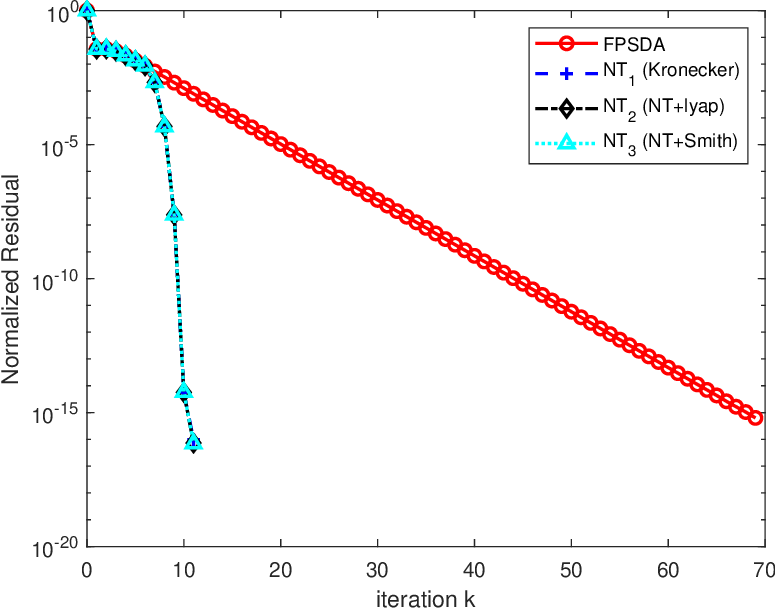}
\caption{Example~\ref{example:3}}
\label{fig:example3__convergence}
\end{subfigure}
\begin{subfigure}[b]{0.42\textwidth}
\center
\includegraphics[width=\textwidth]{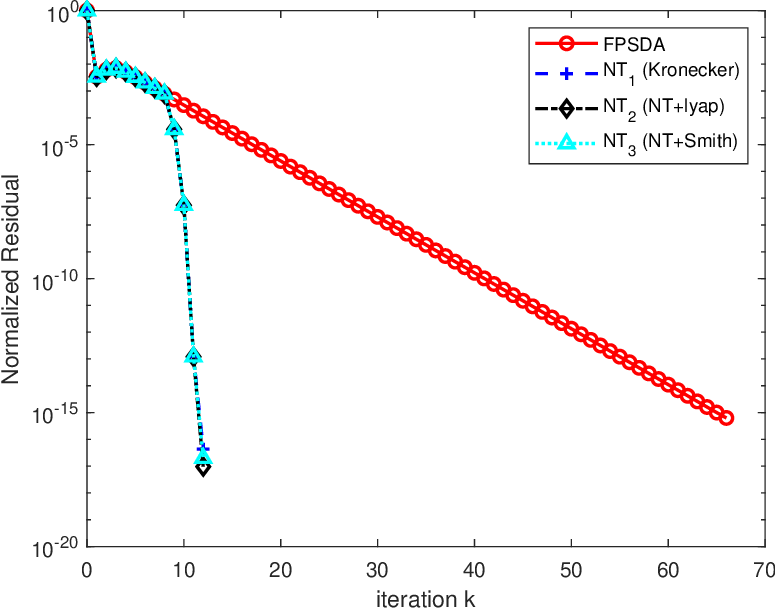}
\caption{Example~\ref{example:F16}}
\label{fig:exampleF16__convergence}
\end{subfigure}\quad
\begin{subfigure}[b]{0.42\textwidth}
\center
\includegraphics[width=\textwidth]{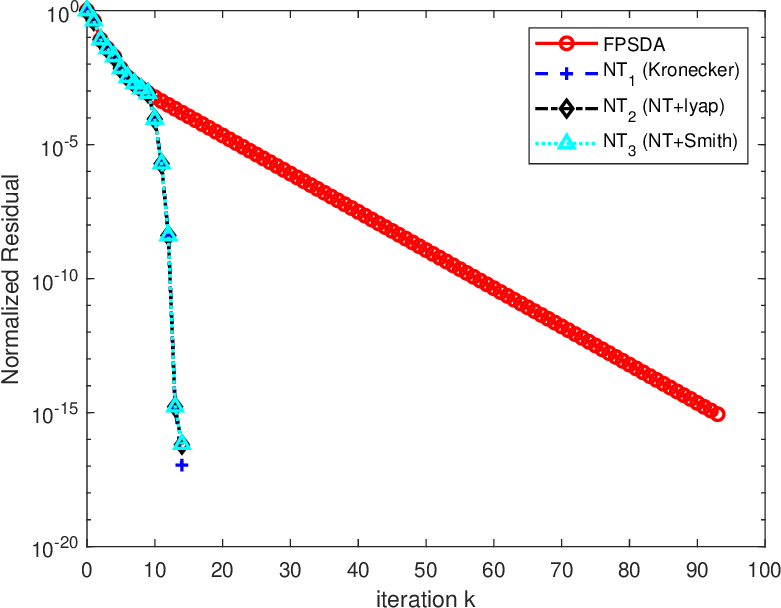}
\caption{Example~\ref{example:Quadrotor}}
\label{fig:exampleQuadrotor__convergence}
\end{subfigure}
\caption{
Iterative histories of \FP\SDA, $\NT_1$ (via Kronecker's reformulation of each Newton step equation),
$\NT_2$ (via MATLAB's {\tt lyap} to iteratively solve Newton step equations),
and $\NT_3$ (via Smith's method to iteratively solve Newton step equations in an outer-inner fashion).
The plotted histories of normalized residual by the variants of Newton's method include the portion by \FP\SDA\
for calculating an initial.
}
\label{fig:hard-Egs}
\end{figure}

\begin{example} \label{example:Quadrotor}
    The \SDRE\ optimal control design for quadrotors for enhancing robustness against unmodeled disturbances \cite{chhu:2022} can be described as in \eqref{eq:DSys-SSDC} and \eqref{eq:quad-func-1} with
    $\bx = [u, v, w, p, q, \varrho, \phi, \theta, z_a]^{\T}$, 
    \begin{align*}
        A(\bx) &= \begin{bmatrix}
            0 & \frac{\varrho}{2} & -\frac{q}{2} & 0 & -\frac{w}{2} & \frac{v}{2} & 0 & -\frac{g(\sin \theta)}{\theta} & 0 \\
            -\frac{\varrho}{2} & 0 & \frac{p}{2} & \frac{w}{2} & 0 & -\frac{u}{2} & \frac{g(1+\cos \theta) \sin \phi}{2\phi} & -\frac{g(1-\cos \theta) \sin \phi}{2\theta} & 0 \\
            \frac{q}{2} & -\frac{p}{2} & 0 & -\frac{v}{2} & \frac{u}{2} & 0 & -\frac{2g}{\phi}\sin^2(\frac{\phi}{2})\cos^2(\frac{\theta}{2}) & -\frac{2g}{\theta} \sin^2(\frac{\theta}{2}) \cos^2(\frac{\phi}{2}) & \frac{g}{z_a} \\
            0 & 0& 0 & 0 & \frac{c_1\varrho}{2} & \frac{c_1q}{2} & 0 & 0 & 0 \\
            0 & 0 & 0 & \frac{c_2\varrho}{2} & 0 & \frac{c_2p}{2} & 0 & 0 & 0 \\
            0 & 0 & 0 & \frac{c_3q}{2} & \frac{c_3p}{2} & 0 & 0 & 0 & 0 \\
            0 & 0 & 0 & 1 & \frac{\sin \phi \tan \theta}{3} & \frac{\tan \theta(1 + \cos \phi)}{3} & \frac{\alpha_1}{3\phi} & \frac{\alpha_2}{3\theta} & 0 \\
            0 & 0 & 0 & 0 & \frac{1+\cos \phi}{2} & -\frac{\sin \phi}{2} & -\frac{\alpha_3}{2\phi} & 0 & 0 \\
            0 & 0 & 0 & 0 & 0 & 0 & 0 & 0 & -\eta
        \end{bmatrix}, \\
        B  &= \begin{bmatrix}
            0 & 0 & -1/\mu & 0 & 0 & 0 & 0 & 0 & 0 \\
            0 & 0 & 0 & {1}/{\cI_x} & 0 & 0 & 0& 0 & 0 \\
            0 & 0 & 0 & 0 & {1}/{\cI_y} & 0 & 0 & 0 & 0 \\
            0 & 0 & 0 & 0 & 0 & {1}/{\cI_z} & 0 & 0 & 0
        \end{bmatrix}^{\T},\\
        Q &= \mbox{\rm diag}(2000, 2000, 3000, 10, 10, 100, 0, 0, 0), \ R = I_4, \ L = 0, \\
        A_0^i &= 0.025 \times i \times  \frac{\| A \|_{\infty}}{\| \widehat{A}_0^i \|_{\infty}} \widehat{A}_0^i\ \text{ with }\ \widehat{A}_0^i = \mbox{\texttt{wgn}}(9,9,10\times i)\,\, \mbox{ for } i = 1, 2, 3, \\
        B_0^i &= 0.01 \times i \times  \frac{\| B \|_{\infty}}{\| \widehat{B}_0^i \|_{\infty}} \widehat{B}_0^i \ \text{ with } \
      \widehat{B}_0^i = \mbox{\texttt{wgn}}(9,4,4\times i)\,\, \mbox{ for } i = 1, 2, 3,
    \end{align*}
    where
    $g$ is the gravity force, $\mu$ is the quadrotor
mass, $(u,v,w)$ and $(p,q,\varrho)$ are the velocity and the angular velocity on the body-fixed frame, respectively, for roll $\phi$, pitch $\theta$ and yaw $\psi$ angles,
\begin{align*}
\alpha_1 &= q\tan \theta \sin \phi - \varrho \tan \theta + \varrho \tan \theta \cos \phi, \\
\alpha_2 &= q \tan \theta \sin \phi + \varrho \tan \theta \cos \phi, \\
\alpha_3 &= q(1-\cos \phi) + \varrho \sin \phi,
\end{align*}
    $z_a$ is a slow varying stable auxiliary variable governed by $\dot{z}_a = -\eta z_a$, $\eta > 0$,
    $c_1 = (\cI_y - \cI_z)/\cI_x$, $c_2 = (\cI_z - \cI_x)/\cI_y$, $c_3 =  (\cI_x - \cI_y)/\cI_z$, and $\cI_x$, $\cI_y$, and $\cI_z$
    are inertia parameters.

With $\mu = 1$,  $\cI_x = \cI_y = 0.01466$, and $\cI_z = 0.02848$, an instance of $A(\bx)$ is
    \begin{small}
    \begin{align*}
        A &= \begin{bmatrix}
            0 & -8.208\mbox{e-}4  & -1.047\mbox{e-}2     &       0 & -1.234\mbox{e-}4 &  1.178  &          0 & -9.8000     &       0 \\
   8.208\mbox{e-}4     &       0  & -1.603\mbox{e-}3 &  1.234\mbox{e-}4    &        0  & 2.203\mbox{e-}2  & 9.800 & -5.436\mbox{e-}4     &       0 \\
   1.047\mbox{e-}2 &  1.603\mbox{e-}3    &        0 & -1.178 & -2.203\mbox{e-}2    &        0    &        0     &       0  &  9.820\mbox{e-}1 \\
            0     &       0     &       0     &       0  & 7.738\mbox{e-}4 &  -9.871\mbox{e-}3  &          0      &      0      &      0 \\
            0      &      0     &       0  & -7.738\mbox{e-}4    &        0  & -1.511\mbox{e-}3 &           0     &       0      &      0 \\
            0      &      0     &       0      &      0     &       0     &       0    &        0     &       0     &       0 \\
            0     &       0     &       0  & 1.000 &  1.386\mbox{e-}8 &  2.499\mbox{e-}4  & 2.617\mbox{e-}6 & -5.464\mbox{e-}4    &        0 \\
            0     &       0     &       0      &      0  & 1.000    &        0  & -9.650\mbox{e-}3      &      0      &      0 \\
            0    &        0    &        0     &       0       &     0      &      0    &        0       &     0  & -0.100
        \end{bmatrix},
    \end{align*}
    \end{small}
yielding an \SCARE\ \eqref{eq:SCARE-0} to solve.
\end{example}

\begin{table}[t]
  \centering
  \begin{tabular}{|l|c|c|c|c|c|c|c|c|c|} \hline
    &  & \multicolumn{2}{c|}{$\NT_1$} & \multicolumn{2}{c|}{$\NT_2$ (\NT +lyap)} & \multicolumn{2}{c|}{$\NT_3$ (\NT +smith)} & \multicolumn{2}{c|}{init. (\FP\SDA)} \\ \cline{3-10}
    & \FP\SDA & \#itn. & sol. err. &  \#itn. & sol. err. & \#itn. & sol. err. & $\delta$ {\scriptsize in \eqref{eq:NT-initial-to-get}} & \#itn.  \\ \hline
\Cref{example:6}    & (18,71) &   &   & (6,25) & $2.8{\scriptstyle (-13)}$  & (6,36,36) & $2.8{\scriptstyle (-13)}$  & $10^{-2}$ & (1,3) \\
\Cref{example:3}   & (69,144) &  5 & $9.5{\scriptstyle (-15)}$ & (5,287)  & $9.8{\scriptstyle (-15)}$ & (5,231,231) & $9.7{\scriptstyle (-15)}$ & $10^{-2}$ & (6,18) \\
\Cref{example:F16} & (66,137) &  4 & $3.3{\scriptstyle (-14)}$  & (4,96) & $3.5{\scriptstyle (-14)}$ & (4,99,99) & $3.5{\scriptstyle (-14)}$ & $10^{-3}$ & (8,21) \\
\Cref{example:Quadrotor} & (93,553) & 5 & $5.8{\scriptstyle (-13)}$ & (5,174) & $6.3{\scriptstyle (-13)}$ & (5,203,203) & $6.4{\scriptstyle (-13)}$ & $10^{-3}$ & (9,49) \\ \hline
  \end{tabular}
  \caption{The numbers of outer and inner iterations  and solution error against the one by \FP\SDA, for solving \SCARE\ \eqref{eq:SCARE-0}.
  For \Cref{example:6}, each Newton's step has a linear system of $199^2$-by-$199^2$ in the standard form to solve and that is too big for most personal PC. That is why there is no result reported for $\NT_1$.
  }
  \label{tab:perf-stat:hard-examples}
\end{table}

As before, \Cref{fig:hard-Egs} plots the iterative histories in terms of normalized residuals
by \FP\SDA\ and three variants of Newton's methods on these four examples, while
\Cref{tab:perf-stat:hard-examples} contains performance statistics of the number of outer iterations and the total number of inner iterations (in each level), and relative solution errors measured as in \eqref{eq:solerr}.

We note again that these examples have random components and hence each call will generate a different \SCARE.
From \Cref{fig:hard-Egs} and \Cref{tab:perf-stat:hard-examples} and from numerous runs on these (random) examples, we observed the following:
\begin{enumerate}[(i)]
  \item With $n=2m-1=199$ in \Cref{example:6}, $\NT_1$ will have to face $199^2$-by-$199^2$ linear systems in the standard form converted from Newton step equations via Kronecker's product. That is too big for everyday PC/laptops such as ours. For that reason, no result is reported on \Cref{example:6} by $\NT_1$.
      On the other hand, all other methods work perfectly well on the example.
  \item \FP\SDA\ converges linearly and the convergence can be slow, as indicated by the large numbers of outer and inner iterations,
        especially for the last three examples.
        Newton's method converges quadratically and consumes very few (outer) numbers of iterations, helped by relatively accurate
        initials by \FP\SDA. But note that both $\NT_2$ and $\NT_3$ requires large number of inner iterations.
  \item For the last three examples, each Newton step equation takes from 24 to 57 fixed-point iterations, i.e., the number of times
        $\what\Pi_k(\cdot)$ in \eqref{eq:NT4SCARE-step} is frozen, but for \Cref{example:6}, the number is between 4 and 6.
        On the other hand, on average, the number of Smith's iteration step is 1.
  \item On average, the number of \SDA\ iterations per \CARE\ in \FP\SDA\ is between 2 and 6.
  \item There are nontrivial chances that \Cref{example:F16,example:Quadrotor} may not have \PSD\ stabilizing solutions as \FP\SDA\
        fail to converge when it is run many times (hence with different random noises). Likely not every random realization of them satisfies
        Assumptions~\ref{asm:stab} and \ref{asm:dete} in \Cref{sec:intro}.
  \item It is reasonable to say that these later four problems are harder than the previous four problems,
        judging from reading the performance statistics in \Cref{tab:perf-stat,tab:perf-stat:hard-examples}.
\end{enumerate}

\section{Conclusions} \label{sec:conclusion}
The state-dependent Riccati equation approach, although suboptimal, is a systematic way to study nonlinear optimal control problems. The basic idea is to formulate a nonlinear optimal control problem into one with
the same linear structure as in the mature linear optimal control theory. Practical applications
have demonstrated its successes. In the approach for systems with stochastic noises,
the so-called stochastic continuous-time algebraic Riccati equation (\SCARE)
$$
A^{\T}X+XA+Q+\Pi_{11}(X)
  -[XB+L+\Pi_{12}(X)][R+\Pi_{22}(X)]^{-1}[XB+L+\Pi_{12}(X)]^{\T}=0,
$$
frequently arises and has to be solved repeatedly, and ideally in real time, e.g., for the 3D missile/target engagement, the F16 aircraft flight control, and the quadrotor optimal control, to name a few. Newton's method had been
called for the task previously, but existing research focuses more on the theoretic aspect than numerical
one, e.g., there is no practical way to pick an initial guess to enure convergence.

In this paper, we propose a robust and efficient inner-outer iterative scheme  to solve the \SCARE\
efficiently and accurately. It is based on
the fixed-point technique and the structure-preserving doubling algorithm (\SDA). The basic idea is to first freeze $X$ in the linear matrix-valued functions $\Pi_{ij}(\cdot)$
at the current approximate solution and then apply \SDA\ to the resulting continuous algebraic Riccati equation
(\CARE). It is proved, among others, that the method is monotonically convergent to the desired stabilizing solution. The new method is called \FP\SDA\ for short.

We revisit Newton's method to investigate  how to calculate each Newton iterative step efficiently and how to select sufficiently accurate initial guesses so that Newton's method can become practical. These are
important issues that have not received much attention previously.

Both methods, ours and Newton's method with our implementation ideas, are applied to a collection of \SCARE, both artificial and real-world ones, to validate our claims and intuitive ideas.

%
%
%

\appendix

\section{Fixed-point iteration}\label{apx:FP}
The basic idea of creating a fixed-point iteration (\FP) is to first transform \SCARE\ \eqref{eq:SCARE-0} equivalently to an equation of the form
\begin{equation}\label{eq:FP-eq}
X=\scrF(X)
\end{equation}
which immediately leads to an iterative scheme:
\begin{equation}\label{eq:FP-general}
\mbox{given $X_0$, iterate
$X_{k+1}=\scrF(X_k)$ for $k=0,1,2,\ldots$
}
\end{equation}
until convergence, if it is convergent.
Guo and Liang~\cite{guli:2023} propose one such $\scrF(\cdot)$ and proved that the associated \FP\ converges.

Guo and Liang~\cite{guli:2023} construct their $\scrF(\cdot)$ as follow.
Factorize $Q - L R^{-1} L^{\T}=\widehat{C}^{\T} \widehat{C}$. In theory, this factorization exists because
$Q - L R^{-1} L^{\T}\succeq 0$. Always $\widehat{C}$ has $n$ columns, but its number of rows is not unique,
except no fewer than the rank of $Q - L R^{-1} L^{\T}\succeq 0$. Numerically, any such a factor $\widehat{C}$ works.
Let
\begin{align*}
      \widehat{A} = A - B R^{-1} L^{\T}, \quad \widehat{B} = B R^{-1/2},
\end{align*}
and let
$P\in \bbR^{nr \times nr}$ and $\what P \in \bbR^{n(r+1) \times n(r+1)}$ be the permutation matrices such that
\begin{align*}
     P^{\T}( X \otimes I_r) P &= I_r \otimes X, \\
     \diag( X, X \otimes I_r ) &= \what P^{\T} ( X \otimes I_{r+1} ) \what P,
\end{align*}
and
\begin{align*}
    \scrA = P \left( \begin{bmatrix}
          A_0^1 \\ \vdots \\ A_0^r
    \end{bmatrix} - \begin{bmatrix}
          B_0^1 \\ \vdots \\ B_0^r
    \end{bmatrix} R^{-1} L^{\T}\right)\in\bbR^{nr\times n}, \quad \scrB = P \begin{bmatrix}
          B_0^1 \\ \vdots \\ B_0^r
    \end{bmatrix} R^{-/12}\in\bbR^{nr\times m}.
\end{align*}
Given $\gamma < 0$ so that $\what A_{\gamma} := \what A + \gamma I_n$ is nonsingular.
\SCARE\ \eqref{eq:SCARE-0} can be transformed into
\begin{equation}\label{eq:guli2023-scrF}
X=\scrF(X):= E_{\gamma}^{\T}(X \otimes I_{r+1}) \left[I_{n(r+1)} + G_{\gamma}(X \otimes I_{r+1})\right]^{-1} E_{\gamma} + H_{\gamma},
\end{equation}
where, with $Z_{\gamma} = \widehat{C} \what A_{\gamma}^{-1} \widehat{B}$,
\begin{subequations} \label{eq:guli2023-scrF-mtx}
\begin{align}
E_{\gamma} &= \what P\begin{bmatrix}
           \what A_{\gamma} - 2 \gamma I_n + \widehat{B} Z_{\gamma}^{\T} \widehat{C} \\
           \sqrt{-2\gamma} (\scrA + \scrB Z_{\gamma}^{\T}\widehat{C})
      \end{bmatrix} (I_n + \what A_{\gamma}^{-1} \widehat{B}Z_{\gamma}^{\T} \widehat{C})^{-1} \what A_{\gamma}^{-1} \in \bbR^{n(r+1) \times n}, \\
H_{\gamma} &= -2 \gamma \what A_{\gamma}^{-\T} \widehat{C}^{\T} ( I_{\ell} + Z_{\gamma} Z_{\gamma}^{\T})^{-1}
              \widehat{C} \what A_{\gamma}^{-1} \in \bbR^{n \times n}, \\
G_{\gamma} &= \what P \begin{bmatrix}
            \sqrt{-2\gamma} \what A_{\gamma}^{-1} \widehat{B} \\
            \scrA \what A_{\gamma}^{-1}\widehat{B} - \scrB
      \end{bmatrix} ( I_m + Z_{\gamma}^{\T} Z_{\gamma})^{-1} \begin{bmatrix}
            \sqrt{-2\gamma} \what A_{\gamma}^{-1} \widehat{B} \\
            \scrA \what A_{\gamma}^{-1} \widehat{B} - \scrB
      \end{bmatrix}^{\T} \what P^{\T} \in \bbR^{n(r+1) \times n(r+1)}.
\end{align}
\end{subequations}
Guo and Liang~\cite{guli:2023} propose to start their \FP\ with $X_0=0$ in \eqref{eq:FP-general} and
prove its convergence for any
$\gamma<0$ such that $\what A_{\gamma}$ is invertible.
Conceivably, the rate of convergence is dependent of the parameter $\gamma$, but it is not clear, even intuitively,
how to choose $\gamma$ for
fast convergence.

Another problematic issue is that $\scrF(\cdot)$ in \eqref{eq:guli2023-scrF} involves the inverse of
$I_{n(r+1)} + G_{\gamma}(X_k \otimes I_{r+1})\in\bbR^{n(r+1) \times n(r+1)}$,
which can be a drag, unless $n(r+1)$ is modest (under $1,000$ or smaller). In what follows, inspired
by the construction of the first standard form (SF1) in \cite{hull:2018}, we will propose another $\scrF$
for \eqref{eq:FP-eq} in what follows.
We begin with pretending $A_{\rmc}(\cdot)$, $G_{\rmc}(\cdot)$, and $H_{\rmc}(\cdot)$ in \eqref{eq:SCARE-CARE-1} are constant matrices
and apply the construction in \cite[section 5.3]{hull:2018} to matrix pencil
\begin{equation}\label{eq:scrHc}
\begin{bmatrix}
A_{\rmc}(X) & -G_{\rmc}(X) \\
-H_{\rmc}(X) & -\big[A_{\rmc}(X)\big]^{\T}
\end{bmatrix} -\lambda I_{2n}.
\end{equation}
Given $\gamma\in\bbR$ such that
$A_{\rmc,\gamma}(X):=A_{\rmc}(X)+\gamma I_n$ is invertible, symbolically following the procedure there, we end up with its SF1 as
\begin{equation}\label{eq:scrHc-SF1}
\begin{bmatrix}
  E_{\rmc}(X) & 0 \\
 -X_{\rmc}(X) & I_n
\end{bmatrix}-\lambda \begin{bmatrix}
                        I & -Y_{\rmc}(X) \\
                        0 & \big[E_{\rmc}(X)\big]^{\T}
                      \end{bmatrix},
\end{equation}
where
\begin{subequations}\label{eq:SF1-scrF-mtx}
\begin{align}
S(X) &=-\big[A_{\rmc,\gamma}(X)\big]^{\T}-H_{\rmc}(X)\,[A_{\rmc,\gamma}(X)]^{-1}G_{\rmc}(X), \\
E_{\rmc}(X) &= I_n+2\gamma [S(X)]^{-1}, \\
X_{\rmc}(X) &= 2\gamma [S(X)]^{-1}H_{\rmc}(X)[A_{\rmc,\gamma}(X)]^{-1}, \\
Y_{\rmc}(X) &=-2\gamma [A_{\rmc,\gamma}(X)]^{-1} G_{\rmc}(X)[S(X)]^{-1}.
\end{align}
\end{subequations}
Finally, the primal equation associated with \eqref{eq:scrHc-SF1} is given by \cite[p.27]{hull:2018}
\begin{equation}\label{eq:SF1-scrF}
X=\scrF(X):=X_{\rmc}(X)+E_{\rmc}(X)\,X[I_n-Y_{\rmc}(X)\,X]^{-1}E_{\rmc}(X)
\end{equation}
which is equivalent to \eqref{eq:SCARE-CARE-0} and, hence, \SCARE\ \eqref{eq:SCARE-0}. We make a couple of comments regarding
the \FP\ based on \eqref{eq:SF1-scrF} as follows.
\begin{enumerate}[(1)]
  \item $X_{k+1}=\scrF(X_k)$ with \eqref{eq:SF1-scrF} is in fact the second approximation by \SDA\  (\Cref{alg:CARE-SDA})
        applied to \eqref{eq:SCARE-CARE-1} with $A_{\rmc}(\cdot)$, $G_{\rmc}(\cdot)$, and $H_{\rmc}(\cdot)$ freezed at $X_k$.
        This connection leads to an intuitive way to construct an effective $\gamma$ that varies from
        one iterative step to another by examining the eigenvalues in $\bbC_-$ of the matrix pencil
        \eqref{eq:scrHc} with frozen  $A_{\rmc}(\cdot)$, $G_{\rmc}(\cdot)$, and $H_{\rmc}(\cdot)$, as we discussed in \Cref{sec:CARE-SDA} below.
  \item Each \FP\ iterative step based on \eqref{eq:SF1-scrF} involves three inverses of $n$-by-$n$ matrices. This compares favorably
        with \FP\ based on \eqref{eq:guli2023-scrF} for large $n$. Here is why. Suppose that the computational complexity
        of inverting an $n$-by-$n$ matrix is $cn^3$ where $c$ is some constant. Then the matrix inversion costs per \FP\ iterative step
        is $3cn^3$ for \eqref{eq:SF1-scrF} and $(r+1)^3cn^3$ for \eqref{eq:guli2023-scrF}. The cost ratio is
        $(r+1)^3/3$ which is $8/3=2.67$ for $r=1$ and grows rather fast as $r$ increases.
\end{enumerate}



%
%
%

\section{Structure-Preserving Doubling Algorithm for \CARE}\label{sec:CARE-SDA}
We review the structure-preserving doubling algorithm (\SDA) for \CARE, which will be tailored to serve as the workhorse for our \Cref{alg:SCARE-SDA}.
For more detail, the reader is referred to \cite[chapter 5]{hull:2018}.
\CARE\ has a couple of equivalent forms, but the one we will be considering is as follows:
\begin{equation}\label{eq:CARE}
A^{\T} X + X A -X G X + H = 0,
\end{equation}
where $A$, $G^{\T}=G$, and $H^{\T}=H$ are in $\bbR^{n\times n}$.

\begin{theorem}[{e.g., \cite[p.330]{zhdg:1995}}]\label{thm:CARE}
Suppose that $G\succeq 0$ and  $H\succeq 0$.
If $(A,G)$  is stabilizable and $(H,A)$ is detectable, then \CARE\ \eqref{eq:CARE} has a unique positive semidefinite (\PSD) solution $X_*$, and, moreover, the solution is stabilizing, i.e., $\eig(A-GX_*)\subset\bbC_-$, the left half of the complex plane.
\end{theorem}

This is essentially Corollary 13.8 of \cite[p.330]{zhdg:1995} which is stated for $G=BB^{\T}$ and $H=C^{\T}C$  along with the condition that $(A,B)$  is stabilizable and $(C,A)$ is detectable.
\Cref{thm:CARE} is equivalent to the corollary because $G\succeq 0$ and  $H\succeq 0$ imply
that $G=BB^{\T}$ and $H=C^{\T}C$ for some $B\in\bbR^{n \times m}$ and $C\in\bbR^{\ell\times n}$, and because
$(A,G)$  is stabilizable if and only if $(A,B)$  is stabilizable, and $(H,A)$ is detectable if and only if
$(C,A)$ is detectable.

\begin{algorithm}[t]
\caption{\SDA\ for \CARE\   \eqref{eq:CARE}}\label{alg:CARE-SDA}
\begin{algorithmic}[1]
    \REQUIRE $A,\, G,\, H\in\bbR^{n\times n}$ with $G^{\T}=G,\, H^{\T}=H$;
    \ENSURE  $X_*$, the last $X_i$, as the computed solution to \eqref{eq:CARE}.
    \hrule\vspace{1ex}
    \STATE pick an appropriate $\gamma<0$;
    \STATE compute $A_{+\gamma} = A +\gamma I$, $A_{-\gamma} = A -\gamma I$, $S=-A_{+\gamma}^{\T}-H A_{+\gamma}^{-1} G$;
    \STATE compute $E_0 =I+2\gamma S^{-\T}$ and then
           \begin{align*}
           X_0 &=2\gamma S^{-1} H A_{+\gamma}^{-1}, \\
           Y_0 &= A_{+\gamma}^{-1} G-A_{+\gamma}^{-1} GS^{-1}( - H A_{+\gamma}^{-1} G - A_{-\gamma}^{\T});
           \end{align*}
    \FOR{$k=0, 1, \ldots,$ until convergence}
         \STATE compute $E_{k+1}= E_k(I_m-Y_kX_k)^{-1}E_k$ and then
         \begin{align*}
         X_{k+1}&= X_k+E_k^{\T}X_k(I_n-Y_kX_k)^{-1}E_k, \\
         Y_{k+1}&= Y_k+E_k(I_n-Y_kX_k)^{-1}Y_kE_k^{\T};
         \end{align*}
    \ENDFOR
    \RETURN last $X_k$ as the computed solution at convergence.
\end{algorithmic}
\end{algorithm}

The unique \PSD\ solution $X_*\succeq 0$ mentioned in \Cref{thm:CARE} can be efficiently calculated by \SDA\ \cite{hull:2018,lixu:2006}, as outlined
in \Cref{alg:CARE-SDA},
for modest $n$ (up to a couple of thousands so that matrix inversions can be efficiently implemented).
Under the conditions of  \Cref{thm:CARE}, \Cref{alg:CARE-SDA} runs without any breakdown for any $\gamma<0$ such that $A_{+\gamma} := A +\gamma I$ is invertible, i.e.,
all inverses at line 5 exist, and the sequence $\{X_k\}_{k=0}^{\infty}$ is monotonically increasing and convergent:
\begin{equation}\label{eq:CARE-SDA-monotone}
0\preceq X_0\preceq X_1 \preceq \cdots \preceq X_k\to X_*,
\end{equation}
and  the convergence is quadratic.
\SDA\ does not need an initial to begin with and $X_0$ at line 3 can be considered as the first approximation by \SDA.
Also $\{Y_k\}_{k=0}^{\infty}$ converges, too, but to the stabilizing solution of the dual \CARE\
of  \eqref{eq:CARE} \cite[chapter 5]{hull:2018}.

\SDA\ starts by selecting a shift $\gamma <0$. In theory, any $\gamma<0$ will work, but some $\gamma$ is better than others for speedy convergence of \SDA. For the fastest asymptotical rate of convergence, the optimal $\gamma$ is given by \cite{hull:2017,hull:2018}
$$
\gamma:=\arg\min_{\lambda\in\eig(A-GX_*)}\left|\frac {\lambda-\gamma}{\lambda+\gamma}\right|.
$$
Finding this optimal $\gamma$ can be time consuming or outright impractical, and hence is not recommended, but usually a suboptimal one is good enough due to the quadratic convergence of \SDA. Here is an idea from \cite{hull:2017} for small $n$ (under ten or even up to a few hundreds) as in most of our later numerical examples:
\begin{enumerate}[(1)]
  \item calculate the $n$ eigenvalues in $\bbC_-$ of
        $$
        \begin{bmatrix}
          A & -G \\
          -H & -A^{\T}
        \end{bmatrix};
        $$
  \item encircle these $n$ eigenvalues by a rectangle $[a,b]\times [-c,c]$ where $a\le b<0$, and set
        \begin{equation}\label{eq:gamma-suboptimal}
        \gamma=\begin{cases}
                    -\sqrt{b^2+c^2}, &\mbox{if $c^2\ge b(a-b)/2$}, \\
                    -\sqrt{ab-c^2}, &\mbox{if $c^2< b(a-b)/2$}.
                    \end{cases}
        \end{equation}
\end{enumerate}
For more guidelines and discussions on how to pick a suboptimal $\gamma$, the reader is referred to \cite{hull:2017}.
Also required on $\gamma$ in \Cref{alg:CARE-SDA} is that $A_{+\gamma}$ is well-conditioned for its inversion.

When $G=0$, \eqref{eq:CARE} reduces to a Lyapunov's equation
\begin{equation}\label{eq:Lyap}
A^{\T} X + X A + H = 0.
\end{equation}
It is well-known that \eqref{eq:Lyap} always has a unique solution $X_*$ if $\eig(A)\subset\bbC_-$,
and moreover $X_*\succeq 0$ if also $H\succeq 0$. Note now the requirements on stabilizability and detectability as in
\Cref{thm:CARE} are no longer needed. As a special case of \eqref{eq:CARE}, Lyapunov's equation \eqref{eq:Lyap}
can be solved by \SDA, too, and the resulting method is much simpler because now $Y_k=0$ for all $k$. Specifically,
the \SDA\ iteration for \eqref{eq:Lyap} becomes
\begin{subequations}\label{eq:SDA4Lyap}
\begin{align}
& E_0=I-2\gamma A_{+\gamma}^{-1},\,\,
 X_0=-2\gamma  A_{+\gamma}^{-\T}H A_{+\gamma}^{-1},\,\,\mbox{and} \label{eq:SDA4Lyap-1}\\
&E_{k+1}=E_k^2,\,\,X_{k+1}= X_k+E_k^{\T}X_kE_k\,\,\,\mbox{for $k\ge 0$}. \label{eq:SDA4Lyap-2}
\end{align}
\end{subequations}
It coincides with
Smith's method, which broadly is for solving Sylvester's equation \cite{smit:1968,wawl:2012}.

It remains to address how to stop the for-loop in \Cref{alg:CARE-SDA} properly. We refer the reader to \cite{hull:2018,xuxl:2012b} for
more discussions in general. In the case of solving \SCARE\ \eqref{eq:SCARE-0} in \Cref{sec:FP-CARE-SDA},
\Cref{alg:CARE-SDA} is used as an inner iteration to calculate
an update to the current approximation of the stabilizing solution to \eqref{eq:SCARE-0} so that the next approximation is more accurate
after the update. For that purpose, there is no need to solve any intermediate \CARE\ fully accurately in the working precision. In fact,
ideally, the update should be calculated with just enough accuracy so that any more accuracy in the update will not help. Also, in
\Cref{sec:FP-CARE-SDA}, $H=\scrR(\wtd X_k)$ at the $k$th iteration.
In view of this discussion,
in our  use of \Cref{alg:CARE-SDA}, we stop the for-loop if
\begin{equation}\label{eq:CARE-STOP-inner}
\|A^{\T} X_k + X_k A -X_k G X_k + H\|_{\F}\le\tau\cdot \|H\|_{\F}, 
\end{equation}
where $0<\tau<1$ is a preselected error-reducing factor. Unfortunately, optimal $\tau$ is problem-dependent. We tested $\tau=1/2,\,1/4,\,1/8,\,1/10$ and found $\tau=1/8$
provide a good balance to achieve small numbers of inner and outer iterations.
We point out two advantages of \eqref{eq:CARE-STOP-inner} over any of the general purposed criteria discussed in \cite{hull:2018,xuxl:2012b}:
\begin{enumerate}[(1)]
  \item \eqref{eq:CARE-STOP-inner} is cheaper and friendlier to use than general stopping criteria discussed in
        \cite{hull:2018,xuxl:2012b} because it does not require calculating
        $\|A\|_{\F},\, \|G\|_{\F},\,\|H\|_{\F}$, and $\|X_k\|_2$, needed for calculating
        the normalized residual of the \CARE\ at the point.
  \item  \eqref{eq:CARE-STOP-inner} usually stops the for-loop much sooner because of fairly large $\tau$.
\end{enumerate}
In the case of Smith's method \eqref{eq:Lyap} as an inner iteration, it is terminated if
\begin{equation}\label{eq:Lyap-STOP-inner}
\|A^{\T} X_k + X_k A + H\|_{\F}\le\tau\cdot \|H\|_{\F}.
\end{equation}

\section{A modified Newton's method}\label{sec:mNewton}
Guo~\cite{guo:2002a} considered a special case of \SCARE\ \eqref{eq:SCARE-0}: $L = 0$ and $B_0^i = 0$ for $i = 1, \ldots, r$.
For the special case, each Newton step
solves a linear matrix equation having exactly the same  form as  \eqref{eq:NT4SCARE-step}.
Guo proposed a modified Newton method: simply freeze the term $\what\Pi_k(\cdot)$ at $X_k$ and solve
\begin{align}\label{eq:mNT4SCARE-step}
\what A_k^{\T} X_{k+1} + X_{k+1} \what A_k + \what\Pi_k(X_k) + M_k = 0
\end{align}
instead. It does get around the difficulty of solving generalized Lyapunov's equations in Newton's method, but loses
the generic quadratic convergence of the method.

This idea of creating a modified Newton method straightforwardly carries over to our more general \SCARE\ \eqref{eq:SCARE-0}, too.
For the ease of reference, we outline the modified Newton's method as in \Cref{alg:mNT-SCARE}. Its implementation is mostly straightforward,
except that when Lyapunov's equation \eqref{eq:mNT4SCARE-step} in $X_{k+1}$ is solved iteratively, by, e.g., Smith's method, there are a couple of comments worthy mentioning for efficient numerical performance:
\begin{enumerate}[(1)]
  \item \eqref{eq:mNT4SCARE-step} should be solved for an update $Z$ to $X_k$ to yield $X=X_k+Z$:
        $$
        A_k^{\T} Z + Z \what A_k+\underbrace{(A_k^{\T} X_k + X_k \what A_k+C_k)}_{=:\wtd C_k}=0,
        $$
        which is solved iteratively, starting from $Z_0=0$, where $C_k=\what\Pi_k(X_k)+M_k$.
  \item  We may use
        $$
        \|\what A_k^{\T} Z_i + Z_i \what A_k+\wtd C_k\|_{\F}
           \le\tau\cdot\|\wtd C_k\|_{\F},
        $$
        knowing that the modified Newton's method is likely linearly convergent, where ideally $\tau$ is just small enough so that, as an approximation
        to the desired solution of  \SCARE\ \eqref{eq:SCARE-0}, $X_k+Z_i$ is about as accurate as
        if $Z_i$ solves the updating equation exactly, but practical $\tau$ around $1/8$ usually works well.
\end{enumerate}

\begin{algorithm}[t]
\caption{Modified Newton's method (\mNT) for solving \SCARE\ \eqref{eq:SCARE-0}} \label{alg:mNT-SCARE}
\begin{algorithmic}[1]
\REQUIRE $A, Q \in \bbR^{n \times n}$, $B, L  \in \bbR^{n \times m}$, $R = R^{\T} \in \bbR^{m \times m}$, $A_0^i \in \bbR^{n \times n}$, $B_0^i \in \bbR^{n \times m}$ for $i = 1, \ldots, r$, initial $X_0^{\T} = X_0$, and  a tolerance $\varepsilon$.
\ENSURE $X_*$, the last $X_k$, as the computed solution to \eqref{eq:SCARE-0}.
        \hrule\vspace{1ex}
\STATE $k = 0$;
%
\WHILE{$\NRes(X_k)>\varepsilon$}
   \STATE $S_k = X_k B + L + \Pi_{12}(X_k)$,  $R_k = \Pi_{22}(X_k) + R$;
   \STATE $\what A_k = A - B R_k^{-1} S_k^{\T}$, $P_k = \begin{bmatrix}
                                                          I \\
                                                          -R_k^{-1}S_k^{\T}
                                                        \end{bmatrix}$,
          $M_k =  P_k^{\T} \begin{bmatrix}
                               Q & L \\ L^{\T} & R
                             \end{bmatrix} P_k$;
   \STATE $C_{k} = P_k^{\T} \Pi(X_{k}) P_k + M_k$;
      \STATE solve $\what A_k^{\T} X + X \what A_k  + C_k  = 0$ for $X$, e.g., by Smith's method, a special case of \SDA\ (see  discussions in \Cref{sec:CARE-SDA} after \Cref{alg:CARE-SDA}) or
             by the Bartels-Stewart method \cite{bast:1972},
              and set $X_{k+1}$ to be the computed solution;
   \STATE Set $k = k + 1$;
\ENDWHILE
\RETURN $X_* = X_k$.
\end{algorithmic}
\end{algorithm}

\subsection{Convergence Analysis} 
\label{sec:conv_mNewton}
Let $\{ X_k \}_{k=0}^{\infty}$ be generated by \mNT\ (\Cref{alg:mNT-SCARE}). A sufficient condition for \Cref{alg:mNT-SCARE} to
be able to generate the entire sequence is $\eig(\what A_k)\subset\bbC_-$ for all $k$.

For the special case: $L = 0$ and $B_0^i = 0$ for $i = 1, \ldots, r$,
Guo established the following convergence theorem, assuming that $X_{k+1}$ is the exact solution of \eqref{eq:mNT4SCARE-step}.

\begin{theorem}[{\cite{guo:2002a}}]\label{thm:mNT-guo2002a}
Consider \SCARE\ \eqref{eq:SCARE-0} with $L = 0$ and $B_0^i = 0$ for $i = 1, \ldots, r$, and let $X_* \succeq 0$ be a solution.
Suppose \Cref{alg:mNT-SCARE} starts with initial $X_0$ that satisfies
$X_0 \succeq X_*$, $A_0 - G_0 X_0$ is stable and $\scrR(X_0) \preceq 0$.
Then
\begin{enumerate}[{\rm (a)}]
  \item $\eig(\what A_k)\subset\bbC_-$ and $\scrR(X_k) \preceq 0$ for all $k\ge 0$,
  \item $X_0 \succeq X_1 \succeq \cdots \succeq X_k \succeq X_*$ for all $k\ge 0$, and
  \item $\lim_{k \to \infty} X_k = X_{*+}$, a maximal solution of \SCARE\ \eqref{eq:SCARE-0}.
\end{enumerate}
\end{theorem}


%
%

Note that, in \Cref{thm:mNT-guo2002a}, $X_{*+}$ is a maximal solution of \SCARE\ \eqref{eq:SCARE-0}.
Going forward, we will assume \eqref{eq:assume-always} as in our convergence analysis for \FP\SDA\ in \Cref{ssec:conv-FPSDA}.
In particular, Assumptions~\ref{asm:stab} and \ref{asm:dete} are assumed and hence \SCARE\ \eqref{eq:SCARE-0} has a unique \PSD\  solution $X_*$, which is stabilizing. Hence, $X_{*+}=X_{*}$ in \Cref{thm:mNT-guo2002a} for the case.

In what follows, we will consider the more general case without requiring $L = 0$ or $B_0^i = 0$ for $i = 1, \ldots, r$. But
we need to assume that $\{ X_k \}_{k=0}^{\infty}$ by \Cref{alg:mNT-SCARE} exists and $X_k \succeq 0$ for all $k$.
Again it is assumed that \eqref{eq:mNT4SCARE-step} is solved exactly.
Under these assumptions, we will show that  $\{ X_k \}_{k=0}^{\infty}$ is monotonically convergent.

By \eqref{eq:SCARE-1b}, \eqref{eq:NT-PiX}, and \eqref{eq:NT-MXk}, and noticing $S_k=X_kB+L_k$, we have
\begin{align}
\what\Pi_k(X_k) + M_k
&=
    \begin{bmatrix}
        I \\ - R_k^{-1} S_k^{\T}
    \end{bmatrix}^{\T}
    \begin{bmatrix}
        Q_k & L_k \\ L_k^{\T} & R_k
    \end{bmatrix}
    \begin{bmatrix}
        I \\ - R_k^{-1} S_k^{\T}
    \end{bmatrix} \nonumber \\
&=Q_k-L_kR_k^{-1} S_k^{\T}-S_kR_k^{-1}L_k^{\T}+S_kR_k^{-1}S_k^{\T} \nonumber \\
&= H_k + X_k B R_k^{-1} B^{\T} X_k, \label{eq:Pi_k+M_k}
\end{align}
where $H_k=H_{\rmc}(X_k)=Q_k-L_kR_k^{-1} L_k^{\T}$.
Hence the modified Newton's iteration equation \eqref{eq:mNT4SCARE-step} is transformed into
\begin{equation}\label{eq:mNT4SCARE-step'}
A_k^{\T} X_{k+1} + X_{k+1} A_k - X_{k+1} G_k X_{k+1} + H_k + (X_{k+1} - X_k) G_k (X_{k+1} - X_k) = 0.
\end{equation}

\begin{lemma} \label{lm:mNT-stablility}
Assume \eqref{eq:assume-always} and suppose $X_k\succeq 0$.
Then $A_k - G_k X_k$ is stable if and only if $X_{k+1} \succeq 0$.
\end{lemma}

\begin{proof}
It follows from \Cref{lm:detectable} that $(H_k, A_k)$ is detectable, where $H_k=A_{\rmc}(X_k)$ and $A_k=A_{\rmc}(X_k)$ as before.
Factorize $H_k = C_k^{\T} C_k$.
By \eqref{eq:NT-AXk} and \eqref{eq:Pi_k+M_k}, we have
\begin{align*}
    (A_k - G_k X_k)^{\T} X_{k+1} + X_{k+1} (A_k - G_k X_k) = - (H_k + X_k G_k X_k).
\end{align*}
Recall $\what A_k=A_k - G_k X_k$.
If $\what A_k$ is stable, then $X_{k+1} \succeq 0$ by \Cref{lm:LyapnovEq-PSD-prop}.

Suppose $X_{k+1} \succeq 0$. We now prove $\eig(A_k - G_k X_k)\subset\bbC-$.
Assume, to the contrary,  that $A_k - G_k X_k$ is not stable. Then there exist $\bbC^n\ni\by \ne 0$ and $\Re(\lambda) \ge 0$
such that $(A_k - G_k X_k) \by = \lambda \by$, which implies
\begin{align*}
    0 \le 2 \Re(\lambda) \by^{\HH} X_{k+1} \by = - (\by^{\HH} H_k \by + \by^{\HH}X_k G_k X_k \by).
\end{align*}
Since $H_k, G_k \succeq 0$, we conclude that  $H_k \by = 0$ and $G_k X_k \by = 0$. Consequently,
\begin{align*}
    \lambda \by = (A_k - G_kX_k) \by = A_k \by, \quad \by \ne 0, \quad \Re(\lambda) \ge 0, \quad \mbox{ and } \quad H_k \by = 0,
\end{align*}
contradicting that the detectability of $(H_k, A_k)$ is detectable.
\end{proof}


Our first convergence result is that $\{ X_k \}_{k=0}^{\infty}$ by \Cref{alg:mNT-SCARE} is monotonically decreasing.

\begin{theorem} \label{thm:mNT-mono-decr}
Assume \eqref{eq:assume-always}.
Suppose that initial $X_0$ satisfies
\begin{equation}\label{eq:cond:mNT-mono-decr}
X_0\succeq X_*,\,\,
\eig(A_0-G_0X_0)\in\bbC_-,\,\,
\scrR(X_0) \preceq 0.
\end{equation}
If $X_k\succeq 0$ for all $k$,
then we have
\begin{enumerate}[{\rm (a)}]
    \item $X_0 \succeq X_1 \succeq \cdots \succeq X_k \succeq X_*$, $\scrR(X_k)\preceq 0$ for $k\ge 0$;
    \item $\lim_{k\rightarrow \infty}X_k=X_{*}$.
\end{enumerate}
\end{theorem}

\begin{proof}
By \Cref{lm:mNT-stablility} and by the assumption that $X_k\succeq 0$ for all $k$, we conclude that
$A_k - G_k X_k$ is stable for all $k\ge 0$.
First we show, by induction, that for  $k\ge 0$
\begin{align}\label{eq5.2}
    X_k \succeq X_{k+1},\ \ X_k\succeq X_* \text{ and }\scrR(X_k)\preceq 0,
\end{align}
and, as a result, item (a) holds.
For $k=0$, we already have $X_0\succeq X_*$ and $\scrR(X_0)\preceq 0$ by assumption. It remains to show
that $X_0 \succeq X_1$. To this end, by \eqref{eq:Pi_k+M_k}, we have
\begin{align*}
( A_0 - G_0 X_0)^{\T}(X_1 - X_0) + ( X_1 - X_0)( A_0 - G_0 X_0) =-\scrR(X_0)\succeq 0.
\end{align*}
Since $A_0 - G_0 X_0$ is stable, by \Cref{lm:LyapnovEq-PSD-prop} we conclude that  $X_0\succeq X_1$.
Suppose now that \eqref{eq5.2} is true for $k=\ell$. We will prove  it  for $k=\ell+1$.
Let $k=\ell$ in \eqref{eq:mNT4SCARE-step'} to get
\begin{align*}
( A_{\ell} - G_{\ell} X_{\ell})^{\T}&(X_* - X_{\ell+1}) + ( X_* - X_{\ell+1})( A_{\ell} - G_{\ell} X_{\ell})  \\
  &= \begin{bmatrix}
        X_*\\
        -I
     \end{bmatrix}^{\T} \Omega(X_{\ell})
     \begin{bmatrix}
        X_*\\
        -I
     \end{bmatrix}+(X_*-X_{\ell})G_{\ell}(X_*-X_{\ell})\\
  &\succeq \begin{bmatrix}
        X_*\\
        -I
     \end{bmatrix}^{\T} \Omega(X_*) \begin{bmatrix}
        X_*\\
     -I  \end{bmatrix}+(X_*-X_{\ell})G_{\ell}(X_*-X_{\ell}) \\
  &=(X_*-X_{\ell})G_{\ell}(X_*-X_{\ell})\\
  &\succeq 0,
\end{align*}
where we have used $X_{\ell}\succeq X_*$ which implies $\Omega(X_{\ell})\succeq \Omega(X_*)$ by \Cref{lm:Omega-incr},
and \Cref{lm:SCARE-equiv} and $X_*$ is a solution to \SCARE\ \eqref{eq:SCARE-0}.
Therefore, $X_{\ell+1}\succeq X_*$ by \Cref{lm:LyapnovEq-PSD-prop}.
Next, we show that $X_{\ell+1}\succeq X_{\ell+2}$ and $\scrR(X_{\ell+1})\preceq 0$.
By the induction hypothesis, $X_{\ell}\succeq X_{\ell+1}$ which implies $\Omega(X_{\ell})\succeq \Omega(X_{\ell+1})$ by \Cref{lm:Omega-incr}, and hence
we have
\begin{align*}
\scrR(X_{\ell+1})&= A_{\ell+1}^{\T} X_{\ell+1} + X_{\ell+1} A_{\ell+1} - X_{\ell+1} G_{\ell+1} X_{\ell+1} + H_{\ell+1} \\
  &=\begin{bmatrix}
        X_{\ell+1}\\
        -I
     \end{bmatrix}^{\T} \Omega(X_{\ell+1}) \begin{bmatrix}
        X_{\ell+1}\\
     -I  \end{bmatrix}\\
  &\preceq \begin{bmatrix}
        X_{\ell+1}\\
        -I
     \end{bmatrix}^{\T} \Omega(X_{\ell})
     \begin{bmatrix}
        X_{\ell+1}\\
     -I  \end{bmatrix} \\
  &= - (X_{\ell+1} - X_{\ell}) G_{\ell} (X_{\ell+1} - X_{\ell}) \qquad(\mbox{by \eqref{eq:mNT4SCARE-step'} for $k=\ell$})\\
  &\preceq 0.
\end{align*}
Finally, letting $k=\ell+1$ in \eqref{eq:mNT4SCARE-step'} yields
\begin{align*}
( A_{\ell+1} - G_{\ell+1} X_{\ell+1}) ^{\T}&(X_{\ell+2} - X_{\ell+1}) + ( X_{\ell+2} - X_{\ell+1})( A_{\ell+1} - G_{\ell+1} X_{\ell+1})  \nonumber\\
     &= \begin{bmatrix}
        X_{\ell+2}\\
        -I
     \end{bmatrix}^{\T} \Omega(X_{\ell+1})
     \begin{bmatrix}
        X_{\ell+2}\\
        -I
     \end{bmatrix}-
     \begin{bmatrix}
        X_{\ell+1}\\
        -I
     \end{bmatrix}^{\T} \Omega(X_{\ell+1}) \begin{bmatrix}
        X_{\ell+1}\\
     -I  \end{bmatrix} \\
&\qquad\qquad+(X_{\ell+2}-X_{\ell+1})G_{\ell+1}(X_{\ell+2}-X_{\ell+1})\\
     &=-\begin{bmatrix}
        X_{\ell+1}\\
        -I
     \end{bmatrix}^{\T} \Omega(X_{\ell+1})
     \begin{bmatrix}
        X_{\ell+1}\\
     -I  \end{bmatrix}\\
&=-\scrR(X_{\ell+1})\succeq 0,
\end{align*}
yielding $X_{\ell+1}\succeq X_{\ell+2}$ by \Cref{lm:LyapnovEq-PSD-prop}. The induction process is completed.

Since the \PSD\ solution $X_*$ of \SCARE\ \eqref{eq:SCARE-0} is unique,
item (b) is a corollary of item (a).
\end{proof}

The first condition in \eqref{eq:cond:mNT-mono-decr} requires that the initial $X_0$ is above the unique \PSD\ solution $X_*$ of
\SCARE\ \eqref{eq:SCARE-0}. Next, we consider the case to start at $X_0$ that is below $X_{*}$.

Given $W\in\bbH^{n\times n}$, define linear operator $\scrF_W$ in $\bbH^{n\times n}$ as
\begin{equation}\label{eq:scrF}
\scrF_W(Z)
 :=\begin{bmatrix}
       -I_n \\
     R_{\rmc}(W)^{-1}\big[L_{\rmc}(W)+WB\big]^{\T}
     \end{bmatrix}^{\T}
     \,\Pi(Z)\,
                               \begin{bmatrix}
                              -I \\
                              R_{\rmc}(W)^{-1}\big[L_{\rmc}(W)+WB\big]^{\T}
                          \end{bmatrix}.
\end{equation}
 Not that if $Z\succeq 0$, then $\scrF_W(Z)\succeq 0$ because $\Pi(Z)\succeq 0$ for any $Z\succeq 0$. But we will need something stronger than this in what follows.

\begin{assumption}\label{asm:diffHc>0}
There exists $\alpha >0$ such that
$
\scrF_{X_{*}}(Z)\succeq \alpha\, Z\,\,
\mbox{for all $Z \succeq0$}.
$
\end{assumption}

Given $W\in\bbH^{n\times n}$, define matrix-valued function $F_W(Y)$ in $\bbH^{n\times n}$ as
\begin{equation}\label{eq:F(Y)}
F_W(Y):=\begin{bmatrix}
       Y \\
       -I_n
     \end{bmatrix}^{\T}
     \Big[\Omega(Y)-\Omega(W)\Big]
     \begin{bmatrix}
       Y \\
       -I_n
     \end{bmatrix}
    \quad\mbox{for $Y\in\bbH^{n\times n}$}.
\end{equation}

\begin{lemma}\label{lm:diffHc>0}
Suppose that Assumption~\ref{asm:diffHc>0} holds.
Given $X_0\preceq X_{*}$,  if $\|X_{*}-X_0\|_2$ is sufficiently small, then
 we have
\begin{equation}\label{eq:suff-close}
F_W(Y)-(Y-W)BR^{-1}B^{\T}(Y-W)\succeq 0\quad\mbox{for $X_0\preceq W\preceq Y\preceq X_{*}$}.
\end{equation}
\end{lemma}

\begin{proof}
The right-hand side of \eqref{eq:scrF} is continuous with respect to $W\in\bbH^{n\times n}$. Hence Assumption~\ref{asm:diffHc>0} ensures
that for sufficiently small $\|W-X_{*}\|_2$, we have
$\scrF_W(Z)\succeq (2\alpha/3)\, Z$ for all $Z \succeq0$.

It can be seen that $\scrF_W(Z)$ is the Fr\`{e}chet  derivative of $F_W$ at $Y=W$ along $Z$.
The first-order expansion of $F_W(Y)$ at $Y=W$ is
\begin{align*}
F_W(Y)&=\left.F_W'(Y)\right|_{Y=W}[Y-W]+O(\|Y-W\|_2^2) \\
   &=\scrF_W(Y-W)+O(\|Y-W\|_2^2).
\end{align*}
If $\|X_{*}-X_0\|_2$ is sufficiently small, so is $\|W-Y\|_2$. Hence for sufficiently small $\|X_{*}-X_0\|_2$, we have
\begin{align*}
F_W(Y)&\succeq (2\alpha/3)\, (Y-W)+O(\|Y-W\|_2^2) \\
   &\succeq (\alpha/3)\, (Y-W),
\end{align*}
Therefore, again for sufficiently small $\|X_{*}-X_0\|_2$, we have
\begin{align*}
F_W(Y)-(Y-W)BR^{-1}B^{\T}(Y-W)
  &\succeq (\alpha/3)\, (Y-W)-\|BR^{-1}B^{\T}\|_2(Y-W)^2 \\
  &=(Y-W)^{1/2}\Big[(\alpha/3) I_n-\|BR^{-1}B^{\T}\|_2(Y-W)\Big](Y-W)^{1/2} \\
  &\succeq 0,
\end{align*}
provided $(\alpha/3) I_n\succeq \|BR^{-1}B^{\T}\|_2(Y-W)$ which can be guaranteed by making $\|X_{*}-X_0\|_2$ sufficiently small.
\end{proof}

The purpose of having \Cref{lm:diffHc>0} is to justify  the requirement in the next theorem on
the sufficient closeness of the initial $X_0$ to $X_{*}$ such that
\eqref{eq:suff-close} holds for any  $X_0\preceq W\preceq Y\preceq X_{*}$. One of the assumption
in both \Cref{thm:mNT-mono-incr,thm:mNT-mono-decr} is that $X_k\succeq 0$ for all $k$, which is upsetting. Ideally,
we should look for other reasonable assumptions that can ensure $X_k\succeq 0$ for all $k$, but we are unable to do so
now.


\begin{theorem}\label{thm:mNT-mono-incr}
Assume \eqref{eq:assume-always}, and
suppose that Assumption~\ref{asm:diffHc>0} holds and that initial $X_0$ satisfies
$$
0\preceq X_0\preceq X_{*},\,\,
\mbox{\eqref{eq:suff-close} holds for any  $X_0\preceq Y\preceq W\preceq X_{*}$, and}\,\,
\scrR(X_0)\succeq 0.
$$
If  $X_k\succeq 0$ for all $k$,
then we have
 \begin{enumerate}[{\rm (a)}]
     \item $0\preceq X_0 \preceq X_1 \preceq \cdots \preceq X_k \preceq X_*$
     and $\scrR(X_k)\succeq 0$ for all $k\ge 0$;
     \item $\lim_{k\rightarrow \infty}X_k=X_{*}$.
 \end{enumerate}
\end{theorem}

\begin{proof}
We prove item (a) by induction. For $k=0$, we already have $X_0\preceq X_{*}$ and $\scrR(X_0)\succeq 0$. By \eqref{eq:Pi_k+M_k}, we have
\begin{align*}
      ( A_0 - G_0 X_0)^{\T}(X_1 - X_0) + ( X_1 - X_0)( A_0 - G_0 X_0) =-\scrR(X_0)\preceq 0.
\end{align*}
Since $A_0-G_0X_0$ is stable, by \Cref{lm:LyapnovEq-PSD-prop}, we obtain that  $X_0\preceq X_1$.
Suppose now that item (a) is true for $k=\ell\ge 0$. We will prove it for $k=\ell+1$.
It can be verified  by \eqref{eq:mNT4SCARE-step'} for $k=\ell$ that
\begin{align*}
( A_{\ell} -  G_{\ell}X_{\ell})^{\T}(X_{*} - X_{\ell+1})
    &+ ( X_{*} - X_{\ell+1})( A_{\ell} - G_{\ell} X_{\ell})  \nonumber\\
    &= \begin{bmatrix}
          X_{*}\\
         -I
       \end{bmatrix}^{\T} \Omega(X_{\ell})
       \begin{bmatrix}
          X_{*}\\
         -I
       \end{bmatrix}+(X_{*}-X_{\ell})G_{\ell}(X_{*}-X_{\ell}) \\
    &= -\begin{bmatrix}
          X_{*}\\
         -I
       \end{bmatrix}^{\T}\Big[\Omega(X_{*})- \Omega(X_{\ell})\Big]
      \begin{bmatrix}
        X_{*}\\
       -I  \end{bmatrix}+(X_{*}-X_{\ell})G_{\ell}(X_{*}-X_{\ell}) \\
    &=-F_{X_{\ell}}(X_*)+(X_{*}-X_{\ell})G_{\ell}(X_{*}-X_{\ell}) \\
    &\preceq 0,
\end{align*}
where we have used \eqref{eq:SCARE-2} and that $X_{*}$ is a solution to \SCARE\ \eqref{eq:SCARE-0} for the second equality,
and \Cref{lm:diffHc>0} and $0\preceq X_{\ell}\preceq X_{*}$ for the last inequality.
%
Therefore, $X_{\ell+1}\preceq X_{*}$ by \Cref{lm:LyapnovEq-PSD-prop}.
Next, we show that $X_{\ell+1}\preceq X_{\ell+2}$ and $\scrR(X_{\ell+1})\succeq 0$. Since
\begin{align*}
\begin{bmatrix}
        X_{\ell+1}\\
        -I
       \end{bmatrix}^{\T}\Omega(X_{\ell})\begin{bmatrix}
        X_{\ell+1}\\
     -I  \end{bmatrix}+(X_{\ell+1} - X_{\ell}) G_{\ell} (X_{\ell+1} - X_{\ell})= 0,
\end{align*}
we have
\begin{align*}
\scrR(X_{\ell+1})
    &= \begin{bmatrix}
        X_{\ell+1}\\
        -I
       \end{bmatrix}^{\T} \Omega(X_{\ell+1})
       \begin{bmatrix}
        X_{\ell+1}\\
        -I
       \end{bmatrix}\\
     &= \begin{bmatrix}
        X_{\ell+1}\\
        -I
       \end{bmatrix}^{\T}\Big[\Omega(X_{\ell+1})-\Omega(X_{\ell})\Big]
        \begin{bmatrix}
           X_{\ell+1}\\
           -I
        \end{bmatrix}- (X_{\ell+1} - X_{\ell}) G_{\ell} (X_{\ell+1} - X_{\ell}) \\
     &=F_{X_{\ell}}(X_{\ell+1})-(X_{\ell+1} - X_{\ell}) G_{\ell} (X_{\ell+1} - X_{\ell}) \\
     &\succeq 0,
\end{align*}
where we have used $X_0\preceq X_{\ell}\preceq X_{\ell+1}\preceq X_{*}$ and \Cref{lm:diffHc>0}.
Finally,
\begin{align*}
( A_{\ell+1} - G_{\ell+1} X_{\ell+1})^{\T}&(X_{\ell+2} - X_{\ell+1}) + ( X_{\ell+2} - X_{\ell+1})( A_{\ell+1} - G_{\ell+1} X_{\ell+1})  \nonumber\\
     &=-\begin{bmatrix}
        X_{\ell+1}\\
        -I
       \end{bmatrix}^{\T} \Omega(X_{\ell+1})
       \begin{bmatrix}
        X_{\ell+1}\\
     -I  \end{bmatrix} \\
&=-\scrR(X_{\ell+1})\preceq 0.
\end{align*}
Hence, $X_{\ell+1}\preceq X_{\ell+2}$ by \Cref{lm:LyapnovEq-PSD-prop}. The induction process is completed.

Since the \PSD\ solution $X_*$ of \SCARE\ \eqref{eq:SCARE-0} is unique,
item (b) is a corollary of item (a).
\end{proof}

\section*{Acknowledgments}
Huang, Kuo, and Lin are supported in part by
NSTC 110-2115-M-003-012-MY3, 110-2115-M-390-002-MY3 and 
112-2115-M-A49-010-, respectively.
Lin is also supported in part by the Nanjing Center for Applied Mathematics.
Li is supported in part by US NSF DMS-2009689.


\bibliographystyle{abbrv}
\bibliography{strings,research_papers}

\end{document}